\documentclass[matprg]{svjour3}

\usepackage{graphicx}   
\usepackage{multirow}
\usepackage{color}
\usepackage{amsmath}
\usepackage{amsfonts}
\usepackage{latexsym}
\usepackage{verbatim}
\usepackage{array}
\usepackage{amsmath}    
\usepackage{amsmath, amssymb, algorithm, algorithmic}

\newcommand{\comments}[1]{}

\newtheorem{assumption}{Assumption}

\newcommand{\beq}{\begin{equation}}
\newcommand{\eeq}{\end{equation}}
\newcommand{\beqa}{\begin{eqnarray}}
\newcommand{\eeqa}{\end{eqnarray}}
\newcommand{\beqas}{\begin{eqnarray*}}
\newcommand{\eeqas}{\end{eqnarray*}}
\newcommand{\bi}{\begin{itemize}}
\newcommand{\ei}{\end{itemize}}
\newcommand{\ba}{\begin{array}}
\newcommand{\ea}{\end{array}}

\newcommand{\nn}{\nonumber}

\def\eqnok#1{(\ref{#1})}

\def\vgap{\vspace*{.1in}}

\newcommand{\bbe}{\mathbb{E}}

\def\Prob{{\hbox{\rm Prob}}}
\newcommand{\bbr}{\mathbb{R}}
\def\w{\omega}

\def\SO{{\cal SO}}

\title{
Conditional gradient type methods for \\
 composite nonlinear and stochastic optimization
\thanks{This work was done while the author was working at the Institute for Research in Fundamental Sciences (P.O.Box: 19395-5746, Tehran, Iran) and supported by a grant from the School of Mathematics.}
}
\date{January, 2016 (Revised: August 2017, December 2017)}
\author{Saeed Ghadimi\thanks{sghadimi@princeton.edu.}
}

\begin{document}
\maketitle
\begin{abstract}
In this paper, we present a conditional gradient type (CGT) method for solving a class of composite optimization problems where the objective function consists of a (weakly) smooth term and a (strongly) convex regularization term. While including a strongly convex term in the subproblems of the classical conditional gradient (CG) method improves its rate of convergence, it does not cost per iteration as much as general proximal type algorithms. More specifically, we present a unified analysis for the CGT method in the sense that it achieves the best known rate of convergence when the weakly smooth term is nonconvex and possesses (nearly) optimal complexity if it turns out to be convex. While implementation of the CGT method requires explicitly estimating problem parameters like the level of smoothness of the first term in the objective function, we also present a few variants of this method which relax such estimation. Unlike general proximal type parameter free methods, these variants of the CGT method do not require any additional effort for computing (sub)gradients of the objective function and/or solving extra subproblems at each iteration. We then generalize these methods under stochastic setting and present a few new complexity results. To the best of our knowledge, this is the first time that such complexity results are presented for solving stochastic weakly smooth nonconvex and (strongly) convex optimization problems.

\end{abstract}

\vgap

\noindent {\bf keywords}
iteration complexity, nonconvex optimization, strongly convex optimization, conditional gradient type methods,
unified methods, weakly smooth functions,


\thispagestyle{plain}

\setcounter{equation}{0}
\section{Introduction}
\label{sec_intro}
In this paper, we study a class of composite nonlinear programming problems given by
\beq \label{NLP}
\Psi^* := \min\limits_{x \in X} \{\Psi(x) := f(x) + h(x)\},
\eeq
where  $X \subseteq \bbr^n$ is a closed convex set, $f:X \to \bbr$ is continuously differentiable but possibly nonconvex {\color{black}such that $\Psi^*$ is finite}, and $h$ is a possibly non-differentiable (strongly) convex function with parameter $\mu\ge0$ with known structure i.e.,
\beq \label{h_strng}
h(\alpha x+(1-\alpha)y) \le \alpha h(x) +(1-\alpha)h(y)-\frac{\alpha(1-\alpha)\mu}{2} \|x-y\|^2
\eeq
for all $x,y \in X$ and $\alpha \in [0,1]$. {\color{black}Note that we allow $\mu=0$ to simply include the case when $h$ is only convex}. Moreover, we assume that $f$ has H\"{o}lder continuous gradient on $X$ i.e.,
there exists $L_\nu>0$ for any $\nu \in (0,1]$ such that
\beq \label{f_holder1}
\|f'(y) - f'(x)\|_* \le L_\nu \|y-x\|^\nu \quad \forall x, y \in X,
\eeq
where $f'(x)$ is the gradient of $f$ at $x$ and $\|\cdot\|_*$ denotes the dual norm.
It can be easily verified that the above inequality implies that
\beq \label{f_holder2}
|f(y) - f(x) - \langle f'(x), y - x \rangle | \le \frac{L_\nu}{1+\nu} \|y - x\|^{1+\nu}
\qquad\forall x, y \in X.
\eeq

The assumption in \eqnok{f_holder1} covers a wide range class of functions including smooth ($\nu = 1$) and
weakly smooth ones ($\nu \in (0,1)$). Different applications arise in the form of problem \eqnok{NLP}. In several machine learning problems, the objective function is given as summation of a loss function and a (strongly) convex regularization term where the loss function can be convex (see e.g., \cite{GraYos05,ZouHas05} ) or nonconvex (see e.g., \cite{MasBaxBarFre99,ChaSinKee08-1}). Signal processing problems using a (strongly) convex regularization term to avoid data overfitting also satisfy in the above setting (see e.g., \cite{JeJoHaJe12}).

Problem \eqnok{NLP} has been well-studied under specific assumptions. Nesterov~\cite{Nest13-1} showed that a variant of his well-known accelerated gradient (AG) method, originally proposed for smooth unconstrained problems, achieves the optimal iteration complexities ${\cal O}(1/\sqrt{\epsilon})$ and ${\cal O}(\ln 1/\epsilon)$ for finding an $\epsilon$-optimal solution of problem \eqnok{NLP} i.e., a solution $\bar x \in X$ s.t. $\Psi(\bar x)-\Psi(x^*) \le \epsilon$, when $f$ is a smooth convex function, $\mu=0$, and $\mu>0$, respectively. Lan~\cite{Lan13-1} presented bundle-level type methods achieving the optimal complexity ${\cal O}(1/\epsilon^\frac{2}{1+3\nu})$ for solving weakly smooth convex problems. Nemirovski and Nesterov~\cite{NemNes85-1} also proposed an optimal proximal gradient type method which possesses the optimal complexity bound ${\cal O}(1/\epsilon^\frac{1-\nu}{1+3\nu})$ for solving weakly smooth strongly convex problems. This bound is also recently obtained in \cite{DeGlNe13,Ito15}. When $f$ is smooth and possibly nonconvex and $h$ is only convex, the (projected) gradient type methods~\cite{CarGouToi10-1,GhaLanZhang14,Nest04} have been proposed achieving the best known complexity ${\cal O}(1/\epsilon)$ to find an $\epsilon$-stationary point of problem \eqnok{NLP} i.e., having $\|g_{_{X,k}}\|^2 \le \epsilon$ for at least one iteration $k$ (see \eqnok{grad_map} for definition of the so-called gradient mapping $g_{_{X,k}}$ at the $k$-th iteration). Recently, Nesterov's AG method has been also generalized for nonconvex problems achieving the aforementioned complexity bound~\cite{GhaLan15,GhaLanZha15}. It seems that for general first-order methods, strong convexity of $h$ does not improve the complexity results when $f$ is possibly nonconvex. Ghadimi et al.~\cite{GhaLanZha15} also generalized the AG method and bundle-level type methods which possess the complexity bound ${\cal O}(1/\epsilon^\frac{1+\nu}{2\nu})$ for finding $\epsilon$-stationary point of weakly smooth nonconvex problems. These methods are also uniformly optimal in the sense that they automatically achieve the optimal complexity bound if the problem turns out to be convex. Such uniform methods have been developed before only for convex programming~\cite{Lan13-1,Nest14}.

All of the above-mentioned methods except the bundle-level type methods, generally known as proximal type methods, require solving subproblems in the form of
\beq \label{subproblem2}
\min_{u \in X} \{\langle f'(x),u \rangle+h(u)+\frac{1}{\gamma} V(u,x)\}
\eeq
at each iteration, where $\gamma$ is a stepsize and $V(u,x)$ is a strongly convex function usually referred to Bregman distance function given by $V(u,x)=\w(u)-\w(x) -\langle w'(x),u-x\rangle$ for a strongly convex so-called prox-function $\w$.

Different class of methods include subproblems in which only linear approximation of the objective function is minimized over a bounded feasible set. The classical conditional gradient (CG) method seems to be the eldest one first proposed by Frank and Wolfe in 1956 \cite{FrankWolfe56-1} and later widely used in the literature due to its low iteration cost (see e.g., \cite{Dunn79,Dunn80,Jaggi13,GarberHazan13,HarJudNem12-1,LussTeb13-1,Lan13-2}). While this class of algorithms generally achieves the iteration complexity of ${\cal O}(1/\epsilon)$ for solving smooth strongly convex problems, some of them exhibit linear rate of convergence under specific conditions on the feasible set (see e.g., \cite{PshDan78,GuMar86,GarberHazan13,Lan13-2}). Recently, Nesterov~\cite{Nest15} showed that the CG method achieves the iteration complexity ${\cal O}(1/\epsilon^\frac{1}{\nu})$ for minimizing weakly smooth convex problems. In particular, he considered the class of problems given by \eqnok{NLP} and showed that by modifying the linear optimization in the subproblems of the CG method to
\beq \label{subproblem}
\min_{u \in X} \{\langle f'(x),u \rangle+h(u)\},
\eeq
we can obtain the complexity bound ${\cal O}(1/\epsilon^\frac{1}{2\nu})$ for solving problem \eqnok{NLP} when $f$ is weakly smooth and convex ($\Psi$ is strongly convex).
{\color{black}Jiang and Zhang~\cite{JiaZha14} also presented an algorithm involving subproblems in the form of \eqnok{subproblem} for solving problem \eqnok{NLP} when $f$ is possibly nonconvex and $h$ is convex. This method exhibits iteration complexity of ${\cal O}(1/\epsilon^\frac{1+\nu}{\nu})$ for finding at least one $k$ such that $g_k \le \epsilon$, where $g_k$, defined in \eqnok{def_gk}, is a different termination criterion for finding approximate stationary points (see our discussion after Algorithm~\ref{alg_CGT} about the relation between $g_k$ and the aforementioned gradient mapping $g_{_{X,k}}$).

Our goal in this paper is to make a bridge between the classical CG method and general first-order methods including proximal type ones.
Note that solving subproblem \eqnok{subproblem2} can be more difficult than the above one due to the choice of $\w$. For example, while subproblem \eqnok{subproblem} has a closed form solution for some choices of $h$, we may need an iterative algorithm to solve subproblem \eqnok{subproblem2} due to the existence of Bregman distance function. On the other hand, since variants of the classical CG method require only solving linear programming in each iteration, their iteration costs can be cheaper than solving subproblem \eqnok{subproblem} in general. It is also worth noting that while subproblems of the bundle-level type methods may consist of objective functions similar to that of \eqnok{subproblem}, their feasible sets at each iteration are changed and hence are more complicated than $X$ (see e.g., \cite{Lan13-1}). Indeed, iteration costs of these methods are computationally more expensive than the proximal type methods.

There are several choices of $h$ where \eqnok{subproblem} is solved more efficiently than \eqnok{subproblem2}. A few examples are as follows.
\begin{itemize}
\item [$\bullet$] $h = I_X$, where $I_X$ denotes the indicator function of the bounded set $X$.

\vgap

\item  [$\bullet$] $h(x) =\lambda \|x\|_1$ or $h(x) =\lambda \|x\|_\infty$ for some $\lambda>0$.

\vgap

\item [$\bullet$] $h(x) = \sum_{i=1}^n x_i \log x_i$ which is strongly convex w.r.t $\|\cdot\|_1$ with parameter $1$ over the standard simplex.

\vgap

\item [$\bullet$] $h(x) = \frac{1}{2}\|x\|_p^2$ which is strongly convex w.r.t $\|\cdot\|_p$ with parameter $p-1$, for $p \in (1,2]$.

\vgap

\item [$\bullet$] $h(x)= \frac{1}{2}\|A\|_{S(p)}^2$, where $\|A\|_{S(p)}=\left(\overset{\min(n,m)}{\underset{i=1}{\sum}} x(A)_i^p\right)^\frac{1}{p}$ is the Schatten $l_p$ norm for a matrix $A$ with singular values $\{x(A)_i\}_{i=1}^{\min(n,m)}$. For any $p \in (1,2]$, $h$ is strongly convex w.r.t the Schatten $l_p$ norm with parameter $p-1$ (\cite{kashTeAm12}).
\end{itemize}

\vgap

It is worth noting that for the first two examples, subproblem \eqnok{subproblem} is reduced to a linear programming by possibly adding more variables and/or linear constraints and hence it is equivalent to the subproblems of the classical CG method. This type of subproblems can be solved very efficiently for special cases of the feasible set like polytopes. Also, for the rest of the examples where $h$ is strongly convex, subproblem \eqnok{subproblem} has a closed form solution over their corresponding feasible sets i.e., simplex for the entropy function and the whole space for both squared $l_p$ norm and the squared Schatten norm. For the latter, it is trivial that if $X$ is any subset of the space containing the solution of the unconstrained case, it also admits that closed form solution. Note also that unlike the classical setting of the CG method, we can relax the boundedness assumption of the feasible set when $h$ is strongly convex.

We should point out that when $h$ is strongly convex and the prox-function is set to $h$, difficulty of solving \eqnok{subproblem} and \eqnok{subproblem2} are comparable. However, the prox-function $\w$ is usually considered to be continuously differentiable which may not be true for $h$ in our setting. For examples, squared $l_p$ norms ($p \in (1,2))$ are not differentiable at points with at least one zero component. Moreover, except for very common choices of strongly convex function $h$ like $\|\cdot\|_2^2$, there seems no convergence analysis for this kind of algorithms when applied to (stochastic) composite weakly smooth nonconvex problem. It is also worth noting the stepsize $\gamma$ in \eqnok{subproblem2} usually depends on the problem parameters. On the other hand, since no stepsize is involved in \eqnok{subproblem}, we can design a parameter free algorithm which is computationally much cheaper per iteration, when $h$ is strongly convex, than those include subproblems in the form of \eqnok{subproblem2} regardless of choice of the prox-function (see Algorithm~\ref{alg_CGT_ls} and the discussion after that for more details).

While the iteration cost of the methods including subproblems in the form of \eqnok{subproblem} is comparable or even better than that of the  classical CG method for several examples of $h$ and $X$ (which some of them are mentioned above), it generally lies between that of having linear programming in the subproblems and general proximal type methods including subproblems in the form of \eqnok{subproblem2}.  As we will show, the same fact holds for the iteration complexity of these methods. {\color{black} However, methods including \eqnok{subproblem} can also achieve the iteration complexity of the general proximal type methods under specific assumptions on problem \eqnok{NLP} which consequently implies a better total computational complexity. For example, assume that $f$ is a smooth nonconvex function, $h$ is the squared $l_p$ norm, $X$ is set to $\bbr^n$, or the unit $l_p$ norm ball. Then, our proposed methods, as summarized in Table~\ref{comp}, achieve iteration complexity of ${\cal O}(1/\epsilon)$ to find an $\epsilon$-approximate stationary point of problem \eqnok{NLP} while their subproblems have closed form solutions. Hence, the total number of (sub)gradient computations is still bounded by ${\cal O}(1/\epsilon)$. On the other hand, a few proximal type algorithms like the (projected) gradient method achieve the same ${\cal O}(1/\epsilon)$ iteration complexity (\cite{GhaLanZha15}). However, when the prox-function in subproblem \eqnok{subproblem2} is set to a continuously differentiable strongly convex function like $l_2$ norm squared, then one needs an iterative algorithm to solve the subproblems. It can be easily verified that if the subproblems are solved within $\epsilon$ accuracy, then the iteration complexity of the (projected) gradient method does not change. Since the objective function of the subproblem in this example is nonsmooth and strongly convex, the (projected) subgradient method can solve it within  ${\cal O}(1/\epsilon)$ number of iterations (subgradient computations of $h$). Therefore, the total number of (sub)gradient computations for this method is bounded by ${\cal O}(1/\epsilon^2)$ which is worse than the above-mentioned ${\cal O}(1/\epsilon)$ complexity for our proposed methods.
}

Our contribution in this paper consists of the following three aspects. First, we present a conditional gradient type (CGT) method for solving problem \eqnok{NLP} which only differs from the classical CG method in incorporating function $h$ into the subproblems as shown in \eqnok{subproblem}. 
We show that the number of iterations performed by the CGT method to find an $\epsilon$-approximate stationary point of problem \eqnok{NLP}, when $g_k$ is used as the termination criterion,
is bounded by
\beq \label{nocvx_best}
{\cal O}\left\{\frac{L_\nu^\frac{1}{\nu}}{\epsilon^\frac{1+\nu}{\nu}}\right\}, \qquad \qquad \qquad
{\cal O}\left\{\left(\frac{L_\nu}{\epsilon}\right)^\frac{1}{\nu} \log\frac{1}{\epsilon} \right\},
\eeq
when $f$ is nonconvex or convex, respectively. In the former case, the first bound is similar to the one obtained in \cite{JiaZha14} for a variant of the CGT method with different stepsizes. In the latter, the above second bound also guarantees finding an $\epsilon$-optimal solution of \eqnok{NLP} and is in the same order, up to a logarithmic factor, of the one obtained in \cite{Nest15} with a different stepsize policy. However, it is still worse than the optimal complexity bound of ${\cal O}(\epsilon^{-\tfrac{2}{1+3\nu}})$ for the general first-order methods. Moreover, if $h$ is strongly convex, then by using a different stepsize policy, the above-mentioned bounds are, respectively, improved to
\beq \label{nocvx_best2}
{\cal O}\left\{\left(\frac{L_\nu^2}{(\mu^\nu \epsilon)^{1+\nu}}\right)^\frac{1}{2\nu}\right\}, \qquad \qquad \qquad
{\cal O}\left\{\left(\frac{L_\nu^2 }{\mu^{\nu(1+\nu)} \epsilon^{1-\nu}}\right)^\frac{1}{2\nu}\log\frac{1}{\epsilon} \right\}.
\eeq
The first bound, disregarding $\mu$, is the best-known iteration complexity achieved by the first-order methods when applied to (composite) weakly smooth nonconvex problems. Recently, Ghadimi et al. \cite{GhaLanZha15} obtained a similar bound for a class of uniformly optimal first-order methods. Note however, that while function $h$ in \eqnok{NLP} is assumed to be only convex in \cite{GhaLanZha15}, their algorithms include subproblems in the form of \eqnok{subproblem2} for the choice of $\w(x)=\frac{1}{2}\|x\|_2^2$ and are more complicated than the CGT method. On the other hand, when $f$ is convex, the second bound is worse than the optimal complexity bound of ${\cal O}(\epsilon^{\tfrac{\nu-1}{1+3\nu}})$ for first-order methods for minimizing the class of weakly smooth strongly convex functions (\cite{NemNes85-1,DeGlNe13,Ito15}). Note however, that if $f$ is smooth ($\nu=1$), it is reduced to the optimal complexity bound ${\cal O} (\log \tfrac{1}{\epsilon})$ for solving smooth strongly convex problems. It should be also mentioned that we have a unified analysis of the CGT method to obtain these bound in the sense that the same stepsize policy is used regardless of the convexity of $f$. However, knowing the strong convexity of $h$ helps us to get better results by modifying the stepsize policy.

Second, while obtaining the above-mentioned complexity bounds requires explicitly estimating problem parameters $L_\nu$, $\nu$, and $\mu$, we also present two variants of the CGT method which relax such estimations. In particular, when $\mu>0$, we equip the CGT method with a line search procedure which only uses the target accuracy as a parameter in its implementation. On the other hand, since there is no stepsize involved in the subproblem \eqnok{subproblem}, this line search does not require any additional effort for computing extra gradients of $f$ and/or solving more subproblems. In this sense, this variant of the CGT method is computationally much cheaper per iteration than the other parameter free algorithms like the ones in \cite{GhaLanZha15,DeGlNe13,Nest14} using subproblems in the form of \eqnok{subproblem2} for solving weakly smooth problems. It is worth noting that such a parameter free algorithm is much more desirable in a black-box optimization where only first-order information of $f$ in \eqnok{NLP} is given through an oracle and its convexity and level of smoothness may not be exactly known. This variant of the CGT method achieves the same complexity bounds in \eqnok{nocvx_best2}. Furthermore, we present another variant of the CGT method which does not require $h$ to be strongly convex, any line search procedure, and the target accuracy in advance. However, its complexity bounds are slightly worse than the aforementioned ones for the CGT method.

}
Finally, we consider problem \eqnok{NLP} under stochastic setting where only noisy first-order information of $f$ is available via
subsequent calls to a stochastic oracle ($\SO$). More specifically, when $\SO$ receives $x_k \in X$ as the input at the $k$-th call,
it outputs stochastic gradient $G(x_k, \xi_k)$, where $\{\xi_k\}_{k \ge 1}$ are random vectors whose distributions $P_k$ are supported on $\Xi_k \subseteq \bbr^d$. The following assumption is made for the stochastic gradients returned by the $\SO$.

\vgap

\begin{assumption} \label{assump_st_grad}
For any $x \in X$ and $k \ge 1$, we have
\beqa
&\mbox{a)}& \, \, \bbe [G(x, \xi_k)] = f'(x), \label{ass1.a} \\
&\mbox{b)} & \, \, \bbe \left[ \|G(x, \xi_k) - f'(x)\|_*^2 \right] \le \sigma^2. \label{ass1.b}
\eeqa

\end{assumption}
{\color{black}
We present a randomized stochastic CGT (RSCGT) method for solving problem \eqnok{NLP} and show that the total number of calls to the $\SO$ performed by this method to have $\bbe[g_R] \le \epsilon$ is bounded by
\beq\label{nocvx_best_st}
{\cal O}\left\{\sigma^2 \left(\frac{L_\nu^{1+\nu}}{\epsilon^{1+3\nu}}\right)^\frac{1}{\nu}\right\}, \qquad \qquad \qquad
{\cal O}\left\{\sigma^2 \left(\frac{L_\nu}{\epsilon^{1+2\nu}}\right)^\frac{1}{\nu}\right\},
\eeq
when $f$ is nonconvex or convex, respectively. Here, the expectation is taken w.r.t $\xi$ and $R$, where $R$ is a discrete random variable whose probability mass function is supported on $\{1,2,\dots, N\}$ for some given iteration limit $N \ge 1$. This type of randomized framework for stochastic optimization first was proposed by Ghadimi and Lan~\cite{GhaLan12-2} and later used in \cite{GhaLanZhang14,GhaLan15,DangLan13-1}, for general stochastic approximation algorithms applying to smooth (un)constrained problems. To the best of our knowledge, the complexity bounds in \eqnok{nocvx_best_st} seems to be the first ones for the conditional gradient type methods when applied to stochastic weakly smooth nonconvex and convex problems. Note that the first bound is reduced to ${\cal O}(1/\epsilon^4)$, when $\nu=1$, which is also recently obtained by \cite{ReSrPoSm16} for the classical CG method ($h \equiv 0$) when applied to stochastic smooth nonconvex problems\footnote{\cite{ReSrPoSm16} was released several months after releasing the first version of this work.}.
Moreover, when $h$ is strongly convex, the above bounds are improved to
\beq\label{nocvx_best_st2}
{\cal O}\left\{\sigma^2 \left(\frac{L_\nu^{1+2\nu}}{(\mu\epsilon)^\frac{1+4\nu}{2}}\right)^\frac{1}{\nu}\right\}, \qquad \qquad \qquad
{\cal O}\left\{\sigma^2 \left(\frac{L_\nu}{(\mu\epsilon)^\frac{1+2\nu}{2}}\right)^\frac{1}{\nu}\right\}.
\eeq
The first bound is better than ${\cal O}(\epsilon^{-\tfrac{(1+\nu)^2}{\nu}})$ obtained in \cite{JiaZha14} (when $h$ is convex) for a stochastic gradient type method including subproblems in the form of \eqnok{subproblem2} in which the Bregman distance function is replaced by  $\|\cdot\|^{1+\nu}$ and under stronger assumption than Assumption~\ref{assump_st_grad}.b). In addition, when $\nu=1$ and $\mu>0$, by using different stepsize policy we can further improve the complexity bounds in \eqnok{nocvx_best_st2} to
\[
{\cal O}\left\{\frac{\sigma^2 L_1}{\mu^2 \epsilon^2}\right\}, \qquad \qquad \qquad
{\cal O}\left\{\frac{\sigma^2 L_1}{\mu^2 \epsilon} \log \frac{1}{\epsilon}\right\}.
\]
Note that the first bound is in the same order of the one obtained in \cite{GhaLanZhang14,GhaLan15} for stochastic gradient type algorithms, and the second one is optimal up to a logarithmic constant for solving stochastic smooth strongly convex optimization problems in terms of dependence on the accuracy $\epsilon$ (see e.g., \cite{nemyud:83,GhaLan13-1}). We should mention that the bounds in are also obtained in \cite{LanZho16} for a method using only linear optimization in subproblems when applied to smooth stochastic optimization.
A brief summary of comparing our results with the ones for conditional gradient type methods when applied to weakly smooth optimization problems is given in Table~\ref{comp}.
\begin{table}[h]
\caption{Complexity results for algorithms including subproblems in the form of \eqnok{subproblem} for finding an $\epsilon$-approximate stationary point (or optimal point when $f$ is convex) of problem \eqnok{NLP} when $\nu \in (0,1]$.}
\vgap
\setlength\extrarowheight{2.5pt}
\centering
\label{comp}
\footnotesize
{\color{black}\begin{tabular}{|c|c|c|c|c|}
\hline
Algorithms&$f$&$h$&$\sigma$&Complexity\\[4pt]
\hline
\cite{JiaZha14}&nonconvex&convex&$0$&${\cal O}(\epsilon^{-\tfrac{1+\nu}{\nu}})$\\[4pt]
\hline
\multirow{2}{*}{\cite{Nest15}}&\multirow{2}{*}{convex}&convex&\multirow{2}{*}{$0$}&${\cal O}(\epsilon^{-\tfrac{1}{\nu}})$\\[4pt]
\cline{3-3}\cline{5-5}
&&strongly convex&  &${\cal O}(\epsilon^{-\tfrac{1}{2\nu}})$\\[4pt]
\hline
\multirow{10}{*}{}&\multirow{4}{*}{nonconvex}&convex&\multirow{2}{*}{$0$}&${\cal O}(\epsilon^{-\frac{1+\nu}{\nu}})$\\[4pt]\cline{3-3}\cline{5-5}
&&strongly convex&&${\cal O}(\epsilon^{-\frac{1+\nu}{2\nu}})$\\[4pt]
\cline{3-5}
&&convex&\multirow{3}{*}{$>0$}&${\cal O}(\epsilon^{-\frac{1+3\nu}{\nu}})$\\[4pt]\cline{3-3}\cline{5-5}
This paper&&strongly convex &&${\cal O}(\epsilon^{-\tfrac{1+4\nu}{2\nu}})$\\[4pt]
\cline{2-5}
&\multirow{4}{*}{}&convex&\multirow{2}{*}{$0$}&${\cal O}(\epsilon^{-\frac{1}{\nu}}\log \tfrac{1}{\epsilon})$\\[4pt]\cline{3-3} \cline{5-5}
&convex&strongly convex &&${\cal O}(\epsilon^{\tfrac{\nu-1}{2\nu}} \log \tfrac{1}{\epsilon})$\\[4pt]
\cline{3-5}
&&convex&\multirow{3}{*}{$>0$}&${\cal O}(\epsilon^{-\tfrac{1+2\nu}{\nu}})$\\[4pt]\cline{3-3} \cline{5-5}
&&strongly convex &&${\cal O}(\epsilon^{-\tfrac{1+2\nu}{2\nu}})$\\[4pt]
\hline
\end{tabular}}
\end{table}

}

Rest of the paper is organized as follows. In Section~\ref{sec_CGT}, we present the CGT method and its variants with establishing their convergence properties for solving problem \eqnok{NLP}. In Section~\ref{sec_SCGT}, we generalize some of our results to the stochastic setting of problem \eqnok{NLP}. Finally, we conclude the paper by some remarks in Section~\ref{conl_rem}.

\vgap

{\bf Notation.}
For a differentiable function $h: \bbr^n \to \bbr$, $h'(x)$ is the gradient of $h$ at $x$. More generally, when
$h$ is a proper convex function, $p_h(x) \in \partial h(x)$ denotes a subgradient of $h$ in the subdifferential set of $h$ at $x$.
For $x \in \bbr^n$ and $y \in \bbr^n$, $\langle x, y \rangle$ is the standard inner product in $\bbr^n$, and $\|\cdot\|$ denotes a general norm unless otherwise is mentioned. For any real number $r$, $\lceil r \rceil$ and $\lfloor r \rfloor$ denote the nearest integer to $r$ from above and below, respectively.
{\color{black} We also denote the diameter of set $X$ (whenever it is assumed to be bounded) by $D_X = \max_{x,y \in X} \|x-y\|$.
When $f$ is convex, an optimal solution of \eqnok{NLP} is denoted by $x^* \in X$ and $\Psi^* = \Psi(x^*)$. ${\cal O}(1)$ also stands for a positive constant.}

\setcounter{equation}{0}
\section{Conditional gradient type methods for composite optimization}\label{sec_CGT}

Our goal in this section is to present a conditional gradient type (CGT) method for solving problem \eqnok{NLP}. Note that the classical CG method cannot be directly applied to this problem since $h$ is not necessarily differentiable. However, it is shown that if $h$ is Lipschitz continuous, then it can be approximated by using some smoothing techniques and hence the CG method can be used for solving problem \eqnok{NLP} when $f$ is smooth and convex~\cite{Lan13-2}. Different from the classical CG method, we minimize summation of the (strongly) convex function $h$ and a linear approximation of function $f$ in the subproblems of the CGT method.

Below we precisely describe the conditional gradient type (CGT) method.

\begin{algorithm} [H]
	\caption{The conditional gradient type (CGT) algorithm}
	\label{alg_CGT}
	\begin{algorithmic}

\STATE Input:
$x_0 \in X$, $\{\alpha_k\}_{k \ge 1} \in (0,1]$.
\STATE 0. Set $y_0 = x_0$ and $k=1$.
\STATE 1. Compute $f'(y_{k-1})$ and set
\beqa
x_k &=& w(f'(y_{k-1})) := \arg\min\limits_{u \in X} \left\{ \langle f'(y_{k-1}), u \rangle
+ h(u) \right\}, \label{def_xk}\\
y_k &=& (1 - \alpha_k) y_{k-1}+ \alpha_k x_k, \label{def_yk}
\eeqa
\STATE 2. Set $k \leftarrow k+1$ and go to step 1.
	\end{algorithmic}
\end{algorithm}

\vgap

We now add a few remarks about the above algorithm. First, note that as mentioned above, the CGT method only differs from the classical CG method in the existence of $h$ in subproblem \eqnok{def_xk}. As we will show, while this may make the subproblems harder to solve, it improves rate of convergence of the CG method when $h$ is strongly convex.
{\color{black}Second, as mentioned in Section~\ref{sec_intro}, there are several examples of $h$ where solving \eqnok{def_xk} is not harder than subproblems of the classical CG methods. Moreover, when $h$ is strongly convex, by properly choosing stepsizes $\alpha_k$, we can relax the boundedness assumption of feasible set $X$ required in the setting of iterative methods minimizing only linear approximation of the objective function at each iteration.
Third, note that Algorithm~\ref{alg_CGT} is essentially the same as the ones proposed  in \cite{Nest15,JiaZha14}. However, we consider different assumptions on the objective function and discuss different choices of stepsize policies for this algorithm which improve its existing complexity bounds for solving both (strongly) convex and nonconvex problems.} Finally, since $f$ is possibly nonconvex in \eqnok{NLP}, we need to define a different termination criterion to establish the convergence of Algorithm~\ref{alg_CGT}. To do so, we can define gradient mapping at point $x_k \in X$ as
\beq \label{grad_map}
g_{_{X,k}} =  g_{_{X}}(w(f'(y_{k-1}))):= y_{k-1}-x_k.
\eeq
{\color{black}Note that the extra notation of $w(\cdot)$ will be helpful in our convergence analysis when we consider the stochastic setting of problem \eqnok{NLP}.} The equivalent definitions of the gradient mapping have been widely used based on the updating step of different algorithms (see e.g., \cite{Nest13-1,GhaLan15}). Note that by the optimality condition of \eqnok{def_xk} and for all $x \in X$, we have
\beq \label{opt_cond}
\langle -[f'(y_{k-1})+p_h(x_k)], x-x_k \rangle \le 0 \to \langle - p_\Psi(x_k) ,  x-x_k \rangle \le \langle f'(y_{k-1})-f'(x_k), x-x_k \rangle.
\eeq
Moreover, if $X$ is bounded, then by Cauchy-Schwarz inequality and  \eqnok{f_holder1}, we have
\beq \label{opt_cond2}
\langle f'(y_{k-1})-f'(x_k), x-x_k \rangle \le
\| f'(y_{k-1})-f'(x_k)\|_* \|x-x_k\| \le L_v D_X \|y_{k-1}-x_k\|^{\nu} = L_v D_X \|g_{_{X,k}}\|^{\nu}.
\eeq
{\color{black}
Hence, if $\|g_{_{X,k}}\| \le [\epsilon/(L_v D_X)]^\frac{1}{\nu}$, then $x_k$ is an $\epsilon$-approximate stationary point of problem \eqnok{NLP} i.e.,
\[
\langle - p_\Psi(x_k) ,  x-x_k \rangle \le \epsilon \ \ \forall x \in X.
\]
When $f$ is convex, the above relation also implies $\Psi(x_k)-\Psi(x^*) \le \epsilon$.

Another termination criterion can be defined as
\beqa \label{def_gk}
g_k = g(w(f'(y_{k-1})))&:=& \langle f'(y_{k-1}), y_{k-1}\rangle + h(y_{k-1})+ \max_{x \in X} \{-\langle f'(y_{k-1}), x \rangle - h(x)\} \nn \\
&=& \langle f'(y_{k-1}), y_{k-1}-x_k \rangle + h(y_{k-1}) - h(x_k) \ge 0,
\eeqa
where the last equality is due to \eqnok{def_xk} and the inequality follows from convexity of $h$ and \eqnok{opt_cond}. Similarly, for any $x \in X$, we have
\beqa
\langle -f'(y_{k-1}), x-y_{k-1} \rangle + h(y_{k-1}) - h(x) &\le& g_k, \nn \\
\langle -f'(y_{k-1}), x-y_{k-1} \rangle + h(y_{k-1}) - h(x) &\le& \langle -p_\Psi(y_{k-1}), x-y_{k-1} \rangle.\label{weaker_stationry}
\eeqa
Observe that if $y_{k-1}$ is a stationary point of problem \eqnok{NLP}, then RHS of the above inequality is nonpositive for all $x \in X$. Hence, setting $x=x_k$, the LHS equals $g_k$ which becomes nonpositive and together with \eqnok{def_gk} imply that $g_k$ is zero. Therefore, $g_k=0$ is a necessary condition for $y_{k-1}$ to be a stationary point. Note, however, that \eqnok{opt_cond} implies that $\|g_{_{X,k}}\|=0$ is a sufficient condition for $x_k$($=y_{k-1}$) to be a stationary point.
It can be also seen that when $h \equiv 0$, both criteria are equivalent (at different points). In this case, $g_k$ will be the Frank-Wolf gap defined in the literature of conditional gradient methods. Hence, it makes sense to consider it as a generalized Frank-Wolf gap when $h$ exists.

On the other hand, when $f$ is convex, by \eqnok{def_gk} we have
\beq \label{gk_cvxbnd}
g_k =  h(y_{k-1})+ \max_{x \in X} \{\langle f'(y_{k-1}), y_{k-1}-x \rangle - h(x)\} \ge h(y_{k-1})+ \max_{x \in X} \{f(y_{k-1})- [f(x)+h(x)]\}
=\Psi(y_{k-1})-\Psi(x^*),
\eeq
which implies that $g_k$ provides an upper bound on the optimality gap $\Psi(y_{k-1})-\Psi(x^*)$.
Next result summarizes the relation between $g_k$ and $g_{_{X,k}}$ . It should be pointed out that a similar result stating the relation between $g_k$ and the gradient mapping for proximal methods have been proposed in \cite{JiaZha14} when $h$ is only convex.
\begin{lemma}\label{lemma_gk-gX}
Let $\{x_k,y_k\}_{k \ge 0}$ be generated by Algorithm~\ref{alg_CGT}. Also, suppose that $g_{_{X,k}}$ and $g_k$ are defined in \eqnok{grad_map} and \eqnok{def_gk}, respectively.
\begin{itemize}
\item [a)] If there exist positive constants $M_f$ and $M_h$ such that $\|f'(x)\|_* \le M_f$ and $\|p_h(x)\|_* \le M_h$ for any $x \in X$, then we have $g_k \le (M_f+M_h) \|g_{_{X,k}}\| \ \ \forall k \ge 1$.

\item [b)]If $h$ is strongly convex with parameter $\mu>0$, then we have $\|g_{_{X,k}}\|^2 \le 2 g_k/\mu \ \ \forall k \ge 1$.
\end{itemize}

\end{lemma}

\begin{proof}
First, note that by \eqnok{def_gk}, convexity of $h$, Cauchy-Schwarz, and triangle inequalities, we have
\[
g_k \le \langle f'(y_{k-1})+p_h(y_{k-1}), y_{k-1}-x_k \rangle \le \|f'(y_{k-1})+p_h(y_{k-1})\|_* \|y_{k-1}-x_k\| \le (\|f'(y_{k-1})\|_*+\|p_h(y_{k-1})\|_* )\|y_{k-1}-x_k\|,
\]
which together with \eqnok{grad_map} and boundedness assumption on (sub)gradients of $f$ and $h$ imply part a).
Second, by \eqnok{def_gk}, strong convexity of $h$, and optimality condition of \eqnok{def_xk}, we have
\[
g_k \ge \langle f'(y_{k-1})+p_h(x_k), y_{k-1}-x_k \rangle + \frac{\mu}{2}\|y_{k-1}-x_k \|^2 \ge \frac{\mu}{2}\|y_{k-1}-x_k \|^2,
\]
which together with \eqnok{grad_map} imply part b).
\end{proof}

\vgap

While, $g_k$ denotes a weaker notation of approximate stationary points than $\|g_{_{X,k}}\|$ when $X$ is bounded and $h$ is only convex, the latter cannot be theoretically guaranteed to be convergent. Furthermore, when $h$ is strongly convex with $\mu >0$, Lemma~\ref{lemma_gk-gX}.b) implies that $g_k$ is an upper bound on $\|g_{_{X,k}}\|^2$. In this case, $g_k=0$ is a necessary and sufficient condition for $y_{k-1}$ to be a stationary point of problem \eqnok{NLP}. Therefore, to include more general cases of $h$ and $X$, we provide convergence analysis of Algorithm~\ref{alg_CGT} and its variants based on $g_k$ throughout the paper.
}

Below we establish main convergence properties of Algorithm~\ref{alg_CGT}.
\begin{theorem} \label{theo_CGT}
Suppose that the sequence $\{x_k,y_k\}_{k \ge 0}$ is generated by Algorithm~\ref{alg_CGT}.
\begin{itemize}
{\color{black}
\item [a)] If stepsizes $\alpha_k \in (0,1]$ are chosen such that
\beq \label{cond_alpha_gk}
\alpha_k^{\nu} \le \frac{(1+\nu) g_k}{4 L_\nu \|x_k-y_{k-1}\|^{1+\nu}}  \ \ \forall k \ge 1,
\eeq
then for any $N \ge 1$, we have
\beq \label{main_nocvx_gk}
\sum_{k=1}^N  \alpha_k g_k  \le \frac{4}{3}[\Psi(x_0)- \Psi^*].
\eeq
If, in addition, $f$ is convex, then we have
\beq\label{main_cvx_gk}
\frac{\Psi(y_N)- \Psi^*}{A_N}+\sum_{k=1}^N  \frac{\alpha_k g_k} {4A_k}
\le  \Psi(x_0)- \Psi^*,
\eeq
where
\beq \label{def_Alpha}
A_0 = 1, \ \
A_N = \prod_{k=1}^N \left(1-\frac{\alpha_k}{2}\right) \ \ \forall N \ge 1.
\eeq

\item [b)] If $h$ is strongly convex with parameter $\mu>0$ and stepsizes $\alpha_k \in (0,1]$  are chosen such that
\beq \label{cond_alpha}
\alpha_k^{\nu} \le \frac{(1+\nu) g_k^\frac{1-\nu}{2}}{4 L_\nu} \left(\frac{\mu}{2} \right)^\frac{1+\nu}{2} \ \ \forall k \ge 1,
\eeq
then for any $N \ge 1$, we still obtain \eqnok{main_nocvx_gk} and \eqnok{main_cvx_gk} when $f$ is nonconvex or convex, respectively.
}
\end{itemize}

\end{theorem}

\vgap

\begin{proof}
We first show part a). Noting \eqnok{f_holder2} and \eqnok{def_yk}, we have
{\color{black}
\beqa
f(y_k) &\le& f(y_{k-1})+\langle f'(y_{k-1}), y_k-y_{k-1} \rangle + \frac{L_\nu}{1+\nu} \|y_k-y_{k-1}\|^{1+\nu} \nn \\
&=& f(y_{k-1})-\alpha_k \left[\langle f'(y_{k-1}), y_{k-1}-x_k \rangle - \frac{L_\nu \alpha_k^{\nu}}{1+\nu} \|x_k-y_{k-1}\|^{1+\nu}\right].\label{prof_hold0}
\eeqa
Also, noting \eqnok{h_strng} and \eqnok{def_yk}, we have
\beq \label{h_strng_2}
h(y_k) \le (1-\alpha_k) h(y_{k-1}) +\alpha_k h(x_k)-\frac{\alpha_k(1-\alpha_k)\mu}{2} \|x_k-y_{k-1}\|^2.
\eeq
Adding the above two inequalities and noting \eqnok{def_gk}, we obtain
\beq \label{prof_hold2_gk}
\Psi(y_k) \le \Psi(y_{k-1})- \alpha_k \left[g_k - \frac{L_\nu \alpha_k^{\nu}}{1+\nu} \|x_k-y_{k-1}\|^{1+\nu}+\frac{(1-\alpha_k)\mu}{2} \|x_k-y_{k-1}\|^2\right],
\eeq
which after re-arranging the terms and noting \eqnok{cond_alpha_gk} imply that
\[
\frac{3\alpha_k g_k}{4} \le \Psi(y_{k-1})-\Psi(y_k) \ \ \forall k \ge 1.
\]
Summing up the above inequalities and noting the fact that $\Psi^* \le \Psi(y_k) \ \ \forall k \ge 0$, we obtain \eqnok{main_nocvx_gk}. Moreover, if $f$ is convex, then by subtracting $\Psi^*$ from both sides of \eqnok{prof_hold2_gk}, we have
\beqa
\Psi(y_k)  - \Psi^* &\le& (1-\frac{\alpha_k}{2})[\Psi(y_{k-1})- \Psi^*]+\frac{\alpha_k}{2}[\Psi(y_{k-1})- \Psi^*] - \alpha_k \left[g_k - \frac{L_\nu \alpha_k^{\nu}}{1+\nu} \|x_k-y_{k-1}\|^{1+\nu}\right] \nn \\
&\le& (1-\frac{\alpha_k}{2})\Psi(y_{k-1})- \alpha_k \left[\frac{g_k}{2} - \frac{L_\nu \alpha_k^{\nu}}{1+\nu} \|x_k-y_{k-1}\|^{1+\nu}\right]
\le (1-\frac{\alpha_k}{2})[\Psi(y_{k-1})-\Psi^*]- \frac{\alpha_k g_k}{4},\label{prof_hold_cvx_gk}
\eeqa
where the second and third inequalities are followed by \eqnok{gk_cvxbnd} and \eqnok{cond_alpha_gk}, respectively.
Dividing both sides of the above inequality by $A_k$, noting \eqnok{def_Alpha}, summing them up, and re-arranging the terms, we obtain \eqnok{main_cvx_gk}.

We now show part b). Suppose that $h$ is strongly convex with parameter $\mu>0$. Then, by \eqnok{grad_map}, Lemma~\ref{lemma_gk-gX}.b), and \eqnok{prof_hold2_gk}, we have
\beq \label{prof_hold2}
\Psi(y_k) \le \Psi(y_{k-1})- \alpha_k \left[g_k - \frac{L_\nu \alpha_k^{\nu}}{1+\nu} \left(\frac{2g_k}{\mu}\right)^\frac{1+\nu}{2} \right],
\eeq
which together with \eqnok{cond_alpha} and similarly to part a), imply \eqnok{main_nocvx_gk} and \eqnok{main_cvx_gk}.
}
\end{proof}

\vgap

Observe that the convergence results in Theorem~\ref{theo_CGT} rely on the choice of stepsizes $\{\alpha_k\}_{k \ge 1}$ which should satisfy \eqnok{cond_alpha_gk} or \eqnok{cond_alpha}. Below we specialize these results for specific choices of stepsizes.
{\color{black}
\begin{corollary} \label{corol_CGT}
Suppose that the sequence $\{x_k,y_k\}_{k \ge 0}$ is generated by Algorithm~\ref{alg_CGT} and a target accuracy $\epsilon>0$ is given.
\begin{itemize}
\item [a)] Let $X$ be bounded and stepsizes $\{\alpha_k\}$ set to
\beq \label{cond_alpha1_gk}
\alpha_k =  \left(\frac{g_k}{g_k+\frac{4 L_\nu}{1+\nu}\|x_k-y_{k-1}\|^{1+\nu}}\right)^{\frac{1}{\nu}}  \ \ \forall k \ge 1,
\eeq
then for any $N \ge 1$, we have
\beq \label{main_nocvx-3_gk}
g_{\hat N} \le \max\left\{\frac{4\left[\Psi(x_0)- \Psi^*\right]}{3N \bar \alpha_{\epsilon, \nu}}, \epsilon \right\},
\eeq
where
\beq \label{k_min}
\hat N = \arg\min_{k \in \{1,\ldots,N\}} g_k,
\eeq
and
\beq \label{def_epsilon_gk}
\bar \alpha_{\epsilon,\nu} = \left(\frac{\epsilon}{\epsilon+\frac{4 L_\nu}{1+\nu}D_X^{1+\nu}}\right)^\frac{1}{\nu}.
\eeq
If, in addition, $f$ is convex, then we have
\beq \label{main_cvx-3_gk}
\Psi(y_{\hat N-1}) - \Psi^* \le g_{\hat N} \le  \max \left\{\frac{4[\Psi(x_0)- \Psi^*](1-\tfrac{\bar\alpha_{\epsilon, \nu}}{2})^N}{\bar\alpha_{\epsilon, \nu}}, \epsilon \right\}.
\eeq

\item [b)] If $h$ is strongly convex with parameter $\mu>0$ and stepsizes $\{\alpha_k\}$ are set to
\beq \label{cond_alpha1}
\alpha_k =  \left(\frac{g_k^\frac{1-\nu}{2}}{g_k^\frac{1-\nu}{2}+\frac{4 L_\nu}{1+\nu}\left(\frac{2}{\mu}\right)^\frac{1+\nu}{2}}\right)^{\frac{1}{\nu}}  \ \ \forall k \ge 1,
\eeq
then we have \eqnok{main_nocvx-3_gk} in which $\bar \alpha_{\epsilon,\nu}$ is replaced by
\beq \label{def_epsilon}
\alpha_{\epsilon,\nu} = \left(\frac{\epsilon^\frac{1-\nu}{2}}{\epsilon^\frac{1-\nu}{2}+\frac{4 L_\nu}{(1+\nu)}\left(\frac{2}{\mu}\right)^\frac{1+\nu}{2}}\right)^\frac{1}{\nu}.
\eeq
With this replacement, we also obtain \eqnok{main_cvx-3_gk} when $f$ is convex.
\end{itemize}
\end{corollary}

\begin{proof}
First, observe that the choice of stepsizes in \eqnok{cond_alpha1_gk} satisfies \eqnok{cond_alpha_gk}. Now, if there exists a $k \in \{1,\ldots,N\}$ such that $g_k \le \epsilon$, then both \eqnok{main_nocvx-3_gk} and \eqnok{main_cvx-3_gk} clearly hold due to \eqnok{k_min}. If this is not the case, then we must have $g_k > \epsilon \ \ k=1,\ldots,N$ which together with  \eqnok{def_epsilon_gk}, and boundedness of $X$ imply that
\beq \label{prof_hold_gk3}
\alpha_k \ge \bar \alpha_{\epsilon, \nu} \ \  k=1,\ldots, N,  \ \ \ \ \sum_{k=1}^N  \alpha_k g_k
\ge \sum_{k=1}^N  \bar \alpha_{\epsilon, \nu} g_k
\ge N \bar \alpha_{\epsilon, \nu} g_{\hat N}.
\eeq
Hence, \eqnok{main_nocvx-3_gk} still holds by noting \eqnok{main_nocvx_gk}. Second, noting \eqnok{def_Alpha} and the first inequality in \eqnok{prof_hold_gk3}, we have
\beqa
A_N &=& \prod_{k=1}^N \left(1-\frac{\alpha_k}{2}\right)  \le \prod_{k=1}^N \left(1-\frac{\bar \alpha_{\epsilon, \nu}}{2}\right) = \left(1-\frac{\bar\alpha_{\epsilon, \nu}}{2}\right)^N, \nn \\
\sum_{k=1}^N \frac{\alpha_k}{2A_k} &=& \sum_{k=1}^N \left[\frac{1}{A_k} - \frac{1-\tfrac{\alpha_k}{2}}{A_k}\right] =  \sum_{k=1}^N \left[\frac{1}{A_k} - \frac{1}{A_{k-1}} \right] = \frac{1}{A_N}-1
\ge \frac{\bar\alpha_{\epsilon, \nu}}{2}\left(1-\frac{\bar\alpha_{\epsilon, \nu}}{2}\right)^{-N} \label{def_Alpha2}.
\eeqa
Combining the above relations with \eqnok{gk_cvxbnd} under convexity of $f$, \eqnok{main_cvx_gk}, and noting that $\Psi(y_N) \ge \Psi^*$, we obtain \eqnok{main_cvx-3_gk}. Third, note that when $\mu>0$, the choice of stepsizes in \eqnok{cond_alpha1} satisfies \eqnok{cond_alpha} and then \eqnok{prof_hold_gk3},  \eqnok{def_Alpha2} hold by replacing $\bar \alpha_{\epsilon,\nu}$ with $ \alpha_{\epsilon,\nu}$ due to \eqnok{def_epsilon}. Hence, \eqnok{main_nocvx-3_gk} and \eqnok{main_cvx-3_gk} still hold with this replacement.

\end{proof}
}

\vgap

We now add a few remarks about the above results.
{\color{black}
First, observe that by \eqnok{def_epsilon_gk}, \eqnok{def_epsilon}, and convexity of the univariate function $x \to x^r (r>1)$ on $(0, \infty)$, we have
\beq \label{alpha_epslog_bnd}
\frac{1}{\bar \alpha_{\epsilon, \nu}} =
\left[1+\frac{2 L_\nu D_X^{1+\nu}}{(1+\nu)\epsilon}\right]^\frac{1}{\nu}
\le 2^\frac{1-\nu}{\nu}\left[1+\left(\frac{2 L_\nu D_X^{1+\nu}}{(1+\nu)\epsilon}\right)^\frac{1}{\nu}\right], \ \
\frac{1}{\alpha_{\epsilon, \nu}} \le 2^\frac{1-\nu}{\nu}\left[1+\left(\frac{4 L_\nu}{(1+\nu)\epsilon^\frac{1-\nu}{2}}\right)^\frac{1}{\nu}\left(\frac{2}{\mu}\right)^\frac{1+\nu}{2}\right],
\eeq
which together with \eqnok{main_nocvx-3_gk}, imply that if stepsizes are set to \eqnok{cond_alpha1_gk}, then the total number of iterations performed by Algorithm~\ref{alg_CGT} to have $ g_k \le \epsilon$ for at least one $k$, is bounded by
\beq \label{nocvx_bnd1_gk}
{\cal O}(1) \max \left\{1, \left(\frac{L_\nu D_X^{1+\nu}}{\epsilon}\right)^\frac{1}{\nu}\right\}\left\{\frac{\Psi(x_0)- \Psi^*}{\epsilon}\right\}.
\eeq
This bound (when the second term is the dominating one) is also obtained in \cite{JiaZha14} for Algorithm~\ref{alg_CGT} with different choice of stepsizes. Moreover, when $f$ is convex, the above complexity bound is improved to
\beq \label{cvx_grad_bnd1_gk}
{\cal O} (1) \max \left\{1, \left(\frac{L_\nu D_X^{1+\nu}}{\epsilon}\right)^\frac{1}{\nu}\right\}
\left\{\log \frac{\Psi(x_0)- \Psi^*}{\epsilon}+ \frac{1}{\nu} \log \frac{L_\nu D_X^{1+\nu}}{\epsilon} \right\},
\eeq
which also guarantees fining a solution $\bar x \in X$ such that $\Psi(\bar x)- \Psi(x^*) \le \epsilon$. This bound, up to a logarithmic factor, is similar to the bound ${\cal O}(\epsilon^{-\tfrac{1}{\nu}})$ bound obtained by Nesterov~\cite{Nest15} for Algorithm~\ref{alg_CGT} with different stepsize policy not depending on the problem parameters. However, it is worse than the optimal complexity bound of ${\cal O} (\epsilon^{-\tfrac{2}{1+3\nu}})$ for general first-order methods applied to weakly smooth convex optimization problems \cite{nemyud:83}.

Second, if $\mu>0$ and stepsizes are set to \eqnok{cond_alpha1_gk}, then the complexity bound \eqnok{nocvx_bnd1_gk} is improved to
\beq \label{nocvx_bnd2_gk}
{\cal O} (1) \max \left\{1, \left(\frac{L_\nu^2 }{\mu^{v(1+\nu)} \epsilon^{1-\nu}}\right)^\frac{1}{2\nu}\right\}\left\{\frac{\Psi(x_0)- \Psi^*}{\epsilon}\right\}.
\eeq
}

This bound (when the second term is the dominating one) is obtained by Ghadimi et al. \cite{GhaLanZha15} for a class of gradient type methods when applied to weakly smooth nonconvex problems. However, their methods do not require the knowledge of problem parameters in advance. It should be also mentioned that when $\nu=1$, \eqnok{nocvx_bnd2_gk} is reduced to the best known complexity ${\cal O}(1/\epsilon)$ for first-order methods when applied to smooth nonconvex problems. Furthermore, observe that if $f$ is convex ($\Psi$ is strongly convex), then the bounds in \eqnok{cvx_grad_bnd1_gk} and \eqnok{nocvx_bnd2_gk} are improved to
{\color{black}
\beq \label{cvx_grad_bnd1}
{\cal O} (1) \max \left\{1, \left(\frac{L_\nu^2 }{\mu^{v(1+\nu)} \epsilon^{1-\nu}}\right)^\frac{1}{2\nu}\right\}
\left\{\log \frac{\Psi(x_0)- \Psi^*}{\epsilon}+ \frac{1}{\nu} \log \frac{L_\nu^2}{\mu^{\nu(1+\nu)}\epsilon^{1-\nu}} \right\},
\eeq
}
which also guarantees fining a solution $\bar x \in X$ such that $\Psi(\bar x)- \Psi(x^*) \le \epsilon$. This bound also much better than the complexity of ${\cal O}(\epsilon^{-\tfrac{1}{2\nu}})$ obtained by Nesterov~\cite{Nest15} for Algorithm~\ref{alg_CGT} for solving strongly convex problem. However, it is worse than the optimal complexity bound ${\cal O} (\epsilon^{\tfrac{\nu-1}{1+3\nu}})$ for minimizing weakly smooth strongly convex problems which is achieved by a few more complicated first-order methods \cite{DeGlNe13,Ito15}. It should be also mentioned that these methods do not consider the composite setting. On the other hand, \eqnok{cvx_grad_bnd1} guarantees linear rate of convergence for Algorithm~\ref{alg_CGT} when $f$ is smooth, convex and $h$ is strongly convex.

{\color{black} Finally, observe that our choices of stepsizes in \eqnok{cond_alpha1_gk} and \eqnok{cond_alpha1} do not depend on the convexity of $f$ which provide unified analysis for Algorithm~\ref{alg_CGT}. Moreover, strong convexity of $h$ significantly improves the existing complexity bounds obtained by conditional gradient methods for both convex and nonconvex problems and relaxes the boundedness assumption of the feasible set. On the other hand, these stepsize policies depend on the problem parameters like $\nu$ and $L_\nu$ which may not be known exactly. Below we present a variant of Algorithm~\ref{alg_CGT} which does not require such knowledge of parameters in advance. It should be mentioned that convergence of this algorithm can be only guaranteed when $h$ is strongly convex.}

\begin{algorithm} [H]
	\caption{The conditional gradient type algorithm with line search}
	\label{alg_CGT_ls}
	\begin{algorithmic}

\STATE Input:
$x_0 \in X$, $\{\alpha_k\}_{k \ge 1} \in (0,1]$, $\gamma \in (0,1)$, and $\delta>0$ .
\STATE 0. Set $y_0 = x_0$ and $k=1$.
\STATE 1. Compute $f'(y_{k-1})$ and set $x_k$ to \eqnok{def_xk}.
\STATE 2. Set $y_k$ to \eqnok{def_yk} with the choice of
\beq \label{cond_alpha_ls}
\alpha_k = \gamma^{t_k}  \ \ \forall k \ge 1,
\eeq
where $t_k \ge 1$ is the smallest integer number such that
{\color{black}
\beq \label{line_search}
\Psi(y_k) \le \Psi(y_{k-1})- \alpha_k g_k + \delta \alpha_k
\eeq
}
holds, where $g_k$ is defined in \eqnok{def_gk}.
\STATE 3. Set $k \leftarrow k+1$ and go to step 1.
	\end{algorithmic}
\end{algorithm}

\vgap
Note that in the framework of the CGT method, $y_k$ is indeed a weighted average of $x_i \ \ i=0,1,\ldots,k$ due to \eqnok{def_yk}. Hence, purpose of the line search procedure defined in \eqnok{cond_alpha_ls} and \eqnok{line_search} is to only find the appropriate choice of $\alpha_k$ to specify these weights and subproblem \eqnok{def_xk} is solve once at each iteration.
Similar procedures have been also used for weakly smooth convex and nonconvex problems in \cite{Nest14,GhaLanZha15}. For each run of these line search procedures, a new subproblem is to be solved and gradient computation of the objective function at a new point may be required. Note however, that our procedure does not require this additional effort and hence it is computationally much cheaper per iteration.

In the next result, we show that the line search procedure defined in Step 2 of Algorithm~\ref{alg_CGT_ls} is run only for a finite number of times and hence the complexity of this algorithm is similar to that of Algorithm~\ref{alg_CGT} when $h$ is strongly convex.

\begin{theorem} \label{theo_CGT_ls}
Suppose that the sequence $\{x_k,y_k\}_{k \ge 0}$ is generated by Algorithm~\ref{alg_CGT_ls} and $h$ is strongly convex with parameter $\mu>0$. Then, the following statements hold.

\begin{itemize}
\item [a)] The procedure \eqnok{line_search} is well-defined in the sense that $t_k$ in \eqnok{cond_alpha_ls} is finite for any $k \ge 1$.
{\color{black}
\item [b)] For any $N \ge 1$, we have
\beq \label{main_nocvx_ls}
 \sum_{k=1}^N  \alpha_k g_k  \le \Psi(x_0)-\Psi^*+\delta \sum_{k=1}^N \alpha_k.
\eeq
If, in addition, $f$ is convex, then for any $N \ge 1$, we have
\beq
\frac{\Psi(y_N)- \Psi^*}{A_N}+\sum_{k=1}^N  \frac{\alpha_k g_k} {2A_k}
\le  \Psi(x_0)- \Psi^* + \delta \sum_{k=1}^N \frac{\alpha_k}{A_k},\label{main_cvx_ls}
\eeq
where $A_k$ is defined in \eqnok{def_Alpha}.
}
\end{itemize}

\end{theorem}

\begin{proof}
We first show part a). Noting Lemma~\ref{lemma_gk-gX}.b) and H\"{o}lder inequality $ab \le a^p/p+b^q/q$ with
$p=\frac{2}{1+\nu}$ and $q =  \frac{2}{1-\nu} $ for $\nu \in (0,1)$, we have
\beqa \label{holder_inq}
\frac{L_\nu  (\alpha_k \|x_k-y_{k-1}\|)^{1+\nu}}{1+\nu}
&=& \alpha_k \left[\frac{L_\nu \alpha_k ^\nu\|x_k-y_{k-1}\|^{1+\nu}}{1+\nu} \left(\frac{(1-\nu)}{2 \delta}\right)^{\frac{1-\nu}{2}}\right]
\left[\frac{2 \delta}{(1-\nu)}\right]^{\frac{1-\nu}{2}} \nn \\
&\le& \frac{\hat L_{\nu, \delta} \alpha_k^\frac{1+3\nu}{1+\nu}}{2} \|x_k-y_{k-1}\|^2 +\delta \alpha_k,
\eeqa
where
\beq \label{def_barL}
\hat L_{\nu,\delta} = \left\{\frac{L_\nu}{\left[\frac{2(1+\nu)\delta}{1-\nu}\right]^{\frac{1-\nu}{2}}}\right\}^{\frac{2}{1+\nu}} \ \ \mbox{and} \ \ \hat L_{1,\delta} = \lim_{\nu \to 1} \hat L_{\nu,\delta} = L_1.
\eeq
Combining the above relation with \eqnok{prof_hold2_gk}, we obtain
{\color{black}
\beq
\Psi(y_k) \le \Psi(y_{k-1})- \alpha_k g_k - \frac{\alpha_k \mu}{2} \left[1-\alpha_k - \frac{\hat L_{\nu,\delta} \alpha_k^\frac{2\nu}{1+\nu}}{\mu} \right] \|x_k-y_{k-1}\|^2 + \delta \alpha_k. \label{prof_ls0}
\eeq
}
Noting $\alpha_k \le \alpha_k^\frac{2\nu}{1+\nu}$ due to the facts that $\nu, \alpha_k \in (0,1]$, if
\beq \label{cond_alpha_ls2}
\alpha_k \le \left(\frac{1}{1+\frac{\hat L_{\nu,\delta}}{\mu}}\right)^\frac{1+\nu}{2\nu}:= \alpha_{\delta, \nu},
\eeq
then \eqnok{prof_ls0} clearly implies \eqnok{line_search}. Observe that the above upper bound on $\alpha_k$ is fixed and less than $1$. Hence, this observation and \eqnok{cond_alpha_ls} imply that the line search procedure defined in \eqnok{line_search} is terminated in finite steps. Moreover, assuming that \eqnok{line_search} holds, part b) follows similarly to the proof of Theorem~\ref{theo_CGT} and hence we skip the details.
\end{proof}

\vgap

{\color{black} Observe that the line search procedure defined in Step 2 of Algorithm~\ref{alg_CGT_ls} is performed up to a constant number $\lceil \log \alpha_{\delta, \nu}/\log \gamma \rceil$ per iteration due to \eqnok{cond_alpha_ls2}. Hence, if $\gamma$ is small, then the line search terminates earlier. However, this may lead to a conservative choice of $\alpha_k$. Therefore, a moderate choice like $\gamma = 0.5$ would perform well in practice.

We specialize convergence of Algorithm~\ref{alg_CGT_ls} in the next result.

\begin{corollary} \label{corol_CGT_ls}
Suppose that the sequence $\{x_k,y_k\}_{k \ge 0}$ is generated by Algorithm~\ref{alg_CGT_ls} and $h$ is strongly convex with parameter $\mu>0$. Then, for any $N \ge 1$, we have
\beq \label{main_nocvx_ls2}
g_{\hat N}   \le  \frac{\Psi(x_0)-\Psi^*}{\gamma N \alpha_{\delta, \nu}}+\delta,
\eeq
where $\hat N$ and $\alpha_{\delta, \nu}$ are defined in \eqnok{k_min} and \eqnok{cond_alpha_ls2}, respectively.

If, in addition, $f$ is convex, then we have
\beq
\Psi(y_{\hat N-1})-\Psi^* \le  g_{\hat N} \le 2\left[ \frac{(1-\tfrac{\gamma \alpha_{\delta,\nu}}{2})^N [\Psi(x_0)-\Psi^*]}{\gamma \alpha_{\delta, \nu}}+  \delta\right]. \label{main_cvx_ls2}
\eeq

\end{corollary}
}
\begin{proof}
Noting that by \eqnok{cond_alpha_ls} and \eqnok{cond_alpha_ls2}, we have $\alpha_k \ge \gamma \alpha_{\delta, \nu}$, the results follow similarly to the proof of Corollary~\ref{corol_CGT} in which $\bar \alpha_{\epsilon, \nu}$ is replaced by $\gamma \alpha_{\delta, \nu}$.
\end{proof}

\vgap

Observe that by  \eqnok{cond_alpha_ls2} and similar to \eqnok{alpha_epslog_bnd}, we have
\[
\alpha_{\delta, \nu}^{-1} \le 2^\frac{1-\nu}{\nu} \left[ 1+ \left(\frac{\hat L_{\nu, \delta}}{\mu}\right)^\frac{1+\nu}{2\nu} \right],\nn
\]
which together with the choice of $\delta = \epsilon/4$, \eqnok{def_barL}, and \eqnok{main_nocvx_ls2} imply that the number of iterations performed by Algorithm~\ref{alg_CGT_ls}, to have $ g_k \le \epsilon$ for at least one $k$ is bounded by \eqnok{nocvx_bnd2_gk} and \eqnok{cvx_grad_bnd1} when $f$ is nonconvex or convex, respectively. However, these results are obtain without knowing problem parameters $\nu$ and $L_\nu$. Also note that only positivity of $\mu$ is to be known and its value is not required in the implementation of Algorithm~\ref{alg_CGT_ls}. It should be also mentioned that, while such a prior knowledge of problem parameters is not required in \cite{GhaLanZha15} for general first-order methods to obtain the similar bound for weakly smooth nonconvex problems, their line search procedure is computationally more expensive than the one in \eqnok{line_search} as mentioned before.

{\color{black}Below we present another variant of Algorithm~\ref{alg_CGT} whose implementation is totally independent of all problem parameters and we do not need to know whether $h$ is strongly convex or fix the target accuracy in advance. We should point out that while we do not need any knowledge about the problem parameters, the complexity bounds of this algorithm are worse than the ones we obtained for Algorithm~\ref{alg_CGT_ls}.}

\vgap

\begin{algorithm} [H]
	\caption{The folded conditional gradient type (FCGT) algorithm}
	\label{alg_MCGT}
	\begin{algorithmic}

\STATE Input:
$x_0 \in X$, $\{\alpha_k\}_{k \ge 1} \in (0,1]$, and $\{\beta_k\}_{k \ge 1} \in (0,1]$.
\STATE 0. Set the initial points $y_0=x_0$ and $k=1$.
\STATE 1. Compute $f'(y_{k-1})$ and set $x_k$ to \eqnok{def_xk}.
\STATE 2. Set $y_k = \arg\min_{x \in \{\hat y_k, \tilde y_k\}} \Psi(x)$, where
\beqa
\hat y_k &=& (1 - \alpha_k) y_{k-1}+ \alpha_k x_k,\nn \\
\tilde y_k &=& (1 - \beta_k) y_{k-1}+ \beta_k x_k. \label{def_yk2}
\eeqa
\STATE 2. Set $k \leftarrow k+1$ and go to step 1.
	\end{algorithmic}
\end{algorithm}

Note that the above algorithm only differs from Algorithm~\ref{alg_CGT} in updating the sequence $\{y_k\}_{k \ge 0}$. In particular, two folded sequences $\{\hat y_k\}_{k \ge 0}$ and $\{\tilde y_k\}_{k \ge 0}$ are generated by taking different convex combination of $y_{k-1}$ and $x_k$ and $y_k$ is set to the one with smaller objective value. This way of updating $y_k$, as shown in the next result, allows us to have stepsize policies which are totally independent of the problem parameters.

\begin{theorem} \label{theo_CGT_pf}
Suppose that the sequence $\{x_k,y_k\}_{k \ge 0}$ is generated by Algorithm~\ref{alg_MCGT} and $X$ is bounded.
\begin{itemize}
{\color{black}
\item [a)] For any $N \ge 1$, we have

\beq \label{main_nocvx_pf_gk}
\sum_{k=1}^N  \beta_k g_k \le \Psi(x_0)-\Psi^* + \bar {\cal C}_\nu \sum_{k=1}^N \beta_k^{1+\nu},
\eeq
where
\beq \label{def_barCnu}
\bar {\cal C}_\nu = \frac{L_\nu D_X^{1+\nu}}{(1+\nu)}.
\eeq
If, in addition, $f$ is convex, then for any $N \ge 1$, we have
\beq
\frac{\Psi(y_N)- \Psi^*}{A_N}+\sum_{k=1}^N  \frac{\alpha_k g_k}{ 2A_k}
\le  \Psi(x_0)- \Psi^*+  \bar {\cal C}_\nu \sum_{k=1}^N  \frac{\alpha_k^{1+\nu}}{A_k}, \label{main_cvx_pf2_gk}
\eeq
where $A_k$ is defined in \eqnok{def_Alpha}.
}
\item [b)] If $h$ is strongly convex with parameter $\mu>0$, then for any $N \ge 1$, we have
\beq \label{main_nocvx_pf}
 \sum_{k=1}^N  \beta_k (1-\beta_k) g_k  \le \Psi(x_0)-\Psi^*+ \frac{\mu {\cal C}_\nu^2}{2} \sum_{k=1}^N \beta_k^{1+2\nu},
\eeq
where
\beq \label{def_Cnu}
{\cal C}_\nu = \frac{L_\nu D_X^\nu}{\mu (1+\nu)}.
\eeq
If, in addition, $f$ is convex and stepsizes are chosen such that
\beq \label{cond_alpha_pf}
\alpha_k \le \frac{6}{7} \ \ \forall k \ge 1,
\eeq
then for any $N \ge 1$, we have
\beq
\frac{\Psi(y_N)- \Psi^*}{A_N}+\sum_{k=1}^N  \frac{\alpha_k g_k}{ 2A_k}
\le  \Psi(x_0)- \Psi^*+  \frac{7\mu {\cal C}_\nu^2}{2}  \sum_{k=1}^N  \frac{\alpha_k^{1+2\nu}}{A_k}. \label{main_cvx_pf2}
\eeq
\end{itemize}

\end{theorem}

\vgap

\begin{proof}
{\color{black}
We first show part a). Noting \eqnok{def_yk2}, the facts that $\Psi(y_k) \le \Psi(\hat y_k)$, $\Psi(y_k) \le \Psi(\tilde y_k)$, boundedness of $X$, similar to
\eqnok{prof_hold2_gk}, and \eqnok{prof_hold_cvx_gk}, we obtain
\beqa
\Psi(y_k) &\le& \Psi(\tilde y_k) \le \Psi(y_{k-1}) -\beta_k g_k  + \frac{L_\nu \beta_k^{1+\nu}}{1+\nu} D_X^{1+\nu}, \nn \\
\Psi(y_k) -\Psi^* &\le& \Psi(\hat y_k) -\Psi^*
\le (1-\frac{\alpha_k}{2})[\Psi(y_{k-1})-\Psi(x^*)]
- \frac{\alpha_k}{2} g_k + \frac{L_\nu \alpha_k^{1+\nu}}{1+\nu} D_X^{1+\nu}.\nn
\eeqa
when $f$ is nonconvex and convex, respectively. Then, \eqnok{main_nocvx_pf_gk} and \eqnok{main_cvx_pf2_gk} are followed similarly to the proof of Theorem~\ref{theo_CGT}.a) by noting \eqnok{def_barCnu}.
}

We now show part b). By the fact that $ab \le a^2/2+b^2/2$, we have
\beq \label{prof_hlod_cvx2}
\frac{L_\nu m_1^{1+\nu}}{1+\nu} \|x_k-y_{k-1}\|^{1+\nu} \le \frac{ (m_2 L_\nu)^2 m_1^{1+2\nu}}{2\mu(1+\nu)^2} \|x_k-y_{k-1}\|^{2\nu} + \frac{\mu m_1}{2m_2^2} \|x_k-y_{k-1}\|^2
\eeq
for some $m_1>0$ and $m_2 \in \bbr$. Combining the above relation with the choice of $m_1 =\beta_k$, $m_2=1$, \eqnok{def_yk2}, and \eqnok{prof_hold2_gk} in which $\alpha_k$ is replaced with $\beta_k$, we have
\beqa
\Psi(y_k) \le \Psi(\tilde y_k) &\le& \Psi(y_{k-1})- \beta_k g_k +\frac{\mu \beta_k^2}{2} \|x_k-y_{k-1}\|^2 +\frac{ L_\nu^2 \beta_k^{1+2\nu}}{2\mu(1+\nu)^2} \|x_k-y_{k-1}\|^{2\nu}\nn \\
&\le& \Psi(y_{k-1})- \beta_k (1-\beta_k) g_k +\frac{ L_\nu^2 \beta_k^{1+2\nu}D_X^{2\nu}}{2\mu(1+\nu)^2},
\eeqa
where the second inequality follows from Lemma~\ref{lemma_gk-gX}.b) and boundedness of $X$.
Summing up the above inequalities and re-arranging the terms, and in the view of \eqnok{def_Cnu}, we obtain \eqnok{main_nocvx_pf}.
Moreover, if $f$ is convex, then by \eqnok{def_yk2}, \eqnok{prof_hlod_cvx2} with the choice of $m_1 =\alpha_k$, $m_2=\sqrt 7$, and similar to \eqnok{prof_hold_cvx_gk}, we obtain
\beq \label{prof_hlod_cvx_pf}
\Psi(y_k)-\Psi^* \le  (1-\frac{\alpha_k}{2})[\Psi(y_{k-1})-\Psi(x^*)]
- \frac{\alpha_k}{2} g_k
-\frac{\alpha_k(6-7\alpha_k)\mu}{14} \|x_k-y_{k-1}\|^2 +\frac{7 L_\nu^2 \alpha_k^{1+2\nu}}{2\mu(1+\nu)^2} \|x_k-y_{k-1}\|^{2\nu}
\eeq
Dividing both sides of the above inequality by $A_k$, summing them up, noting boundedness of $X$, \eqnok{cond_alpha_pf}, and in the view of \eqnok{def_Cnu}, we have \eqnok{main_cvx_pf2}.

\end{proof}

\vgap

In the next result, we specify the choice of stepsizes $\beta_k$ and $\alpha_k$ to specialize the rate of convergence of Algorithm~\ref{alg_MCGT} when applied to nonconvex or (strongly) convex problems defined in \eqnok{NLP}.

\begin{corollary} \label{corol_CGT_pf}
Suppose that the sequence $\{x_k,y_k\}_{k \ge 0}$ is generated by Algorithm~\ref{alg_MCGT} and $X$ is bounded.
\begin{itemize}
{\color{black}
\item [a)]  If stepsizes are set to
\beq \label{cond_beta_pf}
\beta_k = \frac{1}{2N^q} \ \  k = 1,2,\ldots,N
\eeq
for some $q \in (0 ,1)$, then for any $N \ge 1$, we have
\beq \label{main_nocvx_pf2_gk}
g_{\hat N} \le \frac{2 [\Psi(x_0)-\Psi^*]}{N^{1-q}}+ \frac{\bar {\cal C}_\nu}{(2N^q)^{\nu}},
\eeq
where $\hat N$ is defined in \eqnok{k_min}.
If, in addition, $f$ is convex and stepsizes are set
\beq \label{cond_alpha_pf2}
\alpha_1 = \frac67, \ \ \alpha_k = \frac{6}{k+5} \ \ \forall k \ge 2,
\eeq
then for any $N \ge 1$, we have
\beq \label{main_cvx_pf3_gk}
\Psi(y_{\hat N -1})- \Psi^* \le g_{\hat N} \le  \frac{160[\Psi(x_0)-\Psi^*]}{(N+3)^3}+\frac{27 \cdot 6^{\nu} \bar {\cal C}_\nu}{(N+5)^\nu}.
\eeq
}
\item [b)] If $h$ is strongly convex with parameter $\mu>0$ and $\beta_k$ is set to \eqnok{cond_beta_pf}, then we have
\beq \label{main_nocvx_pf2}
g_{\hat N}   \le \frac{4 [\Psi(x_0)-\Psi^*]}{N^{1-q}}+ \frac{\mu {\cal C}_\nu^2}{(2N^q)^{2\nu}}.
\eeq
If, in addition, $f$ is convex and \eqnok{cond_alpha_pf2} holds, then we have
\beq \label{main_cvx_pf4}
\Psi(y_{\hat N -1})- \Psi^* \le g_{\hat N} \le  \frac{160[\Psi(x_0)-\Psi^*]}{(N+3)^3} +  \frac{14 \cdot 6^{2\nu} \mu {\cal C}_\nu^2}{(N+5)^{2\nu}}.
\eeq

\end{itemize}

\end{corollary}

\vgap

\begin{proof}
First, note that by \eqnok{cond_beta_pf}, we have
\begin{align}
& \sum_{k=1}^N \beta_k = \frac{N^{1-q}}{2}, \ \
\sum_{k=1}^N \beta_k (1-\beta_k) = \sum_{k=1}^N \frac{1}{2N^q} \left(1-\frac{1}{2N^q}\right) \ge \sum_{k=1}^N \frac{1}{4N^q}= \frac{N^{1-q}}{4},  \nn \\
& \sum_{k=1}^N \beta_k^{1+\nu} = \sum_{k=1}^N \frac{N^{1-q-q\nu}}{2^{1+\nu}}, \ \
\sum_{k=1}^N \beta_k^{1+2\nu} = \frac{N^{1-q-2q\nu}}{2^{1+2\nu}} , \nn
\end{align}
which together with \eqnok{main_nocvx_pf_gk} and \eqnok{main_nocvx_pf} imply \eqnok{main_nocvx_pf2_gk} and \eqnok{main_nocvx_pf2}, respectively.

Second, noting \eqnok{cond_alpha_pf2} and \eqnok{def_Alpha}, we obtain
\begin{align} \label{Ak_ineq}
& A_k =  \frac47 \prod_{i=2}^k \frac{i+2}{i+5}= \frac{480}{7(k+3)(k+4)(k+5)}, \ \ \ \ \ \ \  \sum_{k=1}^N  \frac{\alpha_k}{2A_k} = \frac{1}{A_N}-1 \ge \frac{(N+3)(N+4)(N+5)}{160}, \nn \\
& \sum_{k=1}^N  \frac{\alpha_k^{1+2\nu}}{A_k} \le \sum_{k=1}^N \frac{7 \cdot 6^{2 \nu} (k+3)(k+4)}{80(k+5)^{2\nu}} \le \frac{6^{2 \nu}}{11} \sum_{k=1}^N (k+4)^{2-2\nu} \le \frac{6^{2 \nu}}{11}  \int_1^{N+1} (x+4)^{2-2\nu} dx \nn \\
&= \frac{6^{2 \nu} [(N+5)^{3-2\nu}-5^{3-2\nu}]}{11(3-2\nu)} \le 6^{2 \nu-1} (N+3)(N+4)(N+5)^{1-2\nu}, \nn \\
& \sum_{k=1}^N  \frac{\alpha_k^{1+\nu}}{A_k} \le \sum_{k=1}^N \frac{7 \cdot 6^{\nu} (k+3)(k+4)}{80(k+5)^{\nu}} \le 6^{\nu-1} (N+3)(N+4)(N+5)^{1-\nu},
\end{align}
which together with \eqnok{main_cvx_pf2_gk} and \eqnok{main_cvx_pf2} imply
\eqnok{main_cvx_pf3_gk} and \eqnok{main_cvx_pf4}, respectively.
\end{proof}

\vgap

{\color{black}
We now add a few remarks about the above results. First, note that rate of convergence of Algorithm~\ref{alg_MCGT} when $f$ is nonconvex depends on the choice of $q$. In particular, \eqnok{main_nocvx_pf2_gk} implies that the total number of iterations required by this method to have $g_k \le \epsilon$ for at least one $k$, is bounded by
\beq \label{nocvx_bnd_pf_gk}
{\cal O}(1) \left\{\left(\frac{\Psi(x_0)-\Psi^*}{ \epsilon}\right)^\frac{1}{1-q}+ \left(\frac{L_\nu D_X^{1+\nu}}{ \epsilon}\right)^\frac{1}{q\nu}\right\}.
\eeq
Observe that the best choice is $q=\tfrac{1}{1+\nu}$ which reduces the above bound to the order of the one in \eqnok{nocvx_bnd1_gk}. However, Algorithm~\ref{alg_MCGT} is assumed to be implemented without knowing any problem parameters and hence \eqnok{nocvx_bnd_pf_gk} would be worse than \eqnok{nocvx_bnd1_gk}. It should be also mentioned that the above observation suggests to confine the range of $q$ and have $q \in [1/2, 1)$ to get better a complexity bound. Moreover, when $f$ is convex, then \eqnok{main_cvx_pf3_gk} implies that the iteration complexity of Algorithm~\ref{alg_MCGT} to find and $\epsilon$-optimal solution of problem \eqnok{NLP} is bounded by
\beq \label{cvx_grad_bnd_pf_gk}
{\cal O}(1) \left\{\left(\frac{\Psi(x_0)-\Psi^*}{\epsilon}\right)^\frac{1}{3}+ \left(\frac{L_\nu D_X^{1+\nu}}{\epsilon}\right)^\frac{1}{\nu}\right\},
\eeq
which is in the same order of \eqnok{cvx_grad_bnd1_gk}. However, the above bound is obtained without knowing any problem parameters. It should be  mentioned that this bound (disregarding the first term) is also obtained by Nesterov in \cite{Nest15} for Algorithm~\ref{alg_CGT} when $f$ is convex. However, our choice of stepsizes and analysis in the MCGT method as a variant of Algorithm~\ref{alg_CGT} is slightly different.

Second, when $f$ is possibly nonconvex and $h$ is strongly convex, \eqnok{main_nocvx_pf2} implies that the bound in \eqnok{nocvx_bnd_pf_gk} is improved to
\beq \label{nocvx_bnd_pf}
{\cal O}(1) \left\{\left(\frac{\Psi(x_0)-\Psi^*}{\epsilon}\right)^\frac{1}{1-q}+ \left(\frac{L_\nu D_X^\nu}{\mu \sqrt \epsilon}\right)^\frac{1}{q\nu}\right\},
\eeq
which is worse than the one in \eqnok{nocvx_bnd2_gk}. Similar to the aforementioned observation, \eqnok{nocvx_bnd_pf} also suggests to confine the range of $q$ to $q \in [1/3, 1)$. Furthermore, when $f$ is convex, then \eqnok{main_cvx_pf2} implies that that the bound in \eqnok{nocvx_bnd_pf_gk} is improved to
\beq \label{cvx_grad_bnd_pf}
{\cal O}(1) \left\{\left(\frac{\Psi(x_0)-\Psi^*}{\epsilon}\right)^\frac{1}{3}+ \left(\frac{L_\nu D_X^\nu}{\mu \sqrt{\epsilon}}\right)^\frac{1}{\nu}\right\}.
\eeq
This bound, when the second term is the dominating one, is also obtained in \cite{Nest15}. Although, the above bound is still worse than the one in \eqnok{cvx_grad_bnd1}, it is obtained without knowing any problem parameter or doing any line search procedure. It is worth noting that in this case, we need boundedness of the feasible set even when $h$ is strongly convex.

On the other hand, if we know the problem parameters, it is possible to modify Algorithm~\ref{alg_MCGT} and obtain better results when $\mu>0$.
In particular, we can properly shrink the feasible set and restart Algorithm~\ref{alg_MCGT} to achieve the complexity bound in \eqnok{cvx_grad_bnd1} with the logarithmic factor. However, it requires knowing problem parameters and solving subproblems of \eqnok{def_xk} over the intersection of the feasible set and a ball around the current search point. Hence, this algorithm has limitation in practice and we skip its presentation for sake of simplicity.
}

\setcounter{equation}{0}
\section{Stochastic Conditional gradient type method for composite optimization}\label{sec_SCGT}

In this section, we would like to generalize the algorithms presented in Section~\ref{sec_CGT} for solving problem \eqnok{NLP} under stochastic setting. In particular, we assume that only noisy information of $f$ is available through the stochastic oracle ($\SO$) satisfying Assumption~\ref{assump_st_grad}. Moreover, we consider $\|\cdot\|$ as the norm associated with the inner product in $\bbr^n$ throughout this section.

Below we present a stochastic variant of Algorithm~\ref{alg_CGT} for solving \eqnok{NLP}.

\begin{algorithm} [H]
	\caption{The randomized stochastic conditional gradient type (RSCGT) algorithm}
	\label{alg_SCGT}
	\begin{algorithmic}
\STATE Input: $x_0 \in X$, $\{\alpha_k\}_{k \ge 1} \in (0,1]$, an iteration limit $N \ge 1$, probability mass function (PMF) $P_R(\cdot)$ supported on $\{1,2,\ldots,N\}$, and sequence of positive integer numbers $\{b_k\}_{k \ge 1}$.

\vgap

\STATE 0. Set $y_0 = x_0$, $k=1$, and generate random variable $R$ according to $P_R(\cdot)$.

\vgap

{\bf For $k=1, 2, \ldots, R$:}

{\addtolength{\leftskip}{0.2in}
\STATE 1. Call the stochastic oracle $b_k$ times to compute
\beq \label{def_barG}
\bar G(y_{k-1},\xi_k) = \frac{\sum_{i=1}^{b_k} G(y_{k-1},\xi_{i,k})}{b_k}.
\eeq
\STATE 2. Set
\beq \label{def_xk_tilde}
x_k = w(\bar G(y_{k-1},\xi_k)) = \arg\min\limits_{u \in X} \left\{ \langle \bar G(y_{k-1},\xi_k), u \rangle
+ h(u) \right\},
\eeq
and $y_k$ to \eqnok{def_yk}.

}

{\bf End}

	\end{algorithmic}
\end{algorithm}

Note that the above algorithm differs from Algorithm~\ref{alg_CGT} in two aspects. First, replacing the true gradient of $f$ by its stochastic approximation given in \eqnok{def_barG}. While this average of stochastic gradients is an unbiased estimator for the true gradient of $f$, it has reduced variance with respect to the individual ones. More specially, denoting $\Delta_{i,k}= G(y_{k-1},\xi_{i,k})-f'(y_{k-1})$ and under Assumption~\ref{assump_st_grad}, we have
\beqa
\bbe[\bar G(y_{k-1},\xi_k)] &=& \frac{\sum_{i=1}^{b_k} \bbe[G(y_{k-1},\xi_{i,k})]}{b_k}=f'(y_{k-1}), \nn \\
\bbe \left[\|\bar G(y_{k-1},\xi_k)-f'(y_{k-1})\|_*^2 \right] &=& \frac{1}{b_k^2} \bbe \left[\|\sum_{i=1}^{b_k} \Delta_{i,k}\|_*^2 \right] \le \frac{\sigma^2}{b_k}\label{var_red},
\eeqa
where the last inequality follows from the fact that
$
\bbe \left[\langle \sum_{i=1}^{b_k-1} \Delta_{i,k} ,\Delta_{b_k,k}\rangle | \sum_{i=1}^{b_k-1} \Delta_{i,k} \right]=0.
$

Second, Algorithm~\ref{alg_SCGT} is terminated according to the random variable $R$ since $\hat N$ in \eqnok{k_min} is not computable under stochastic setting. Hence, this algorithm consists of two levels of randomness: $\xi$ coming from the stochastic oracle and $R$ imposed for terminating the algorithm.
This randomization scheme is the current practice for establishing rate of convergence of stochastic algorithms for nonconvex problems. As mentioned in Section~\ref{sec_intro}, it was first proposed by Ghadimi and Lan in \cite{GhaLan12-2} for smooth unconstrained stochastic optimization and later used for constrained and composite optimization (see e.g., \cite{GhaLanZhang14,GhaLan15,DangLan13-1}).

{\color{black}
Noting subproblem \eqnok{def_xk_tilde}, we can define stochastic generalized Frank-Wolf gap ($\tilde g_k$) which will not be useful in our analysis. Hence, we still use \eqnok{def_gk} with $w(f'(y_{k-1}))$ to establish convergence of Algorithm~\ref{alg_SCGT}. Note however, that this quantity is not available under stochastic setting and it is only used to provide convergence analysis. Moreover, we cannot generalize the analysis used in Theorem~\ref{theo_CGT} for the weakly smooth stochastic problems due to the dependence of the stepsizes on $\tilde g_k$. Therefore, we use different stepsize policies to establish convergence analysis of Algorithm~\ref{alg_SCGT}. It is also worth noting that, when  $\nu=1$ and $h$ is strongly convex, then stepsizes similar to those used in Theorem~\ref{theo_CGT} only depend on problem parameters and give us a better rate of convergence. Hence, in addition to the general results, we also present specific rates of convergence when $\nu=1$ and $\mu>0$.

In the next result, we provide main convergence analysis of Algorithm~\ref{alg_SCGT}.

\begin{theorem} \label{theo_SCGT}
Suppose that the sequence $\{x_k,y_k\}_{k \ge 0}$ is generated by Algorithm~\ref{alg_SCGT} and $X$ is bounded.

\begin{itemize}
\item [a)] Assume that $D_X \ge 1$ (for simplicity) and let $R$ be a discrete random variable whose PMF is given by
\beq \label{def_prob_gk}
P_R(R=k) = \frac{\alpha_k}{\sum_{k=1}^N \alpha_k} \ \ k=1,\ldots,N
\eeq
for some given iteration limit $N \ge1$. Then, we have
\beq \label{main_nocvx_st_gk}
\bbe_{\xi_{[N]},R} \left[g_R\right] \le  \frac{\Psi(x_0)- \Psi^*+\sum_{k=1}^N \left[\frac{2L_\nu D_X^2 \alpha_k^{1+\nu}}{(1+\nu)} + \frac{(1+\nu)\sigma^2 \alpha_k^{1-\nu}}{4L_\nu b_k}\right]}{\sum_{k=1}^N \alpha_k}.
\eeq
If $f$ is convex and $P_R(\cdot)$ is given by
\beq \label{def_prob_cvx_gk}
P_R(R=k) = \frac{\alpha_k A_k^{-1} }{\sum_{k=1}^N \alpha_k A_k^{-1}} \ \ k=1,\ldots,N,
\eeq
then we have
\beq
\bbe_{\xi_{[N]},R} \left[\Psi(y_{R-1})-\Psi^*\right] \le \bbe_{\xi_{[N]},R} [g_R]
\le  \frac{2\left\{\Psi(x_0)- \Psi(x^*)+ \sum_{k=1}^N \left[\frac{2L_\nu D_X^2 \alpha_k^{1+\nu}}{(1+\nu) A_k} + \frac{(1+\nu)\sigma^2 \alpha_k^{1-\nu}}{4L_\nu b_k A_k}\right]\right\}}{\sum_{k=1}^N \alpha_k  A_k^{-1}}, \label{main_cvx_st_gk}
\eeq
where $A_k$ is defined in \eqnok{def_Alpha}.

\item [b)] Assume that $h$ is strongly convex with parameter $\mu>0$, and stepsizes are chosen such that
\beq \label{cond_alpha_pf_st}
\alpha_k \le \frac{6}{7} \ \ \forall k \ge 1.
\eeq
If $P_R(\cdot)$ is given by \eqnok{def_prob_gk}, then we have
\beq \label{main_nocvx_pf-st-0}
\bbe_{\xi_{[N]},R} \left[g_R\right] \le  \frac{\Psi(x_0)- \Psi^*+\sum_{k=1}^N \left[\frac{7\mu{\cal C}_\nu^2 \alpha_k^{1+2\nu}}{2} + \frac{\sigma^2 \alpha_k}{2\mu b_k}\right]}{\sum_{k=1}^N \alpha_k},
\eeq
where ${\cal C}_\nu$ is defined in \eqnok{def_Cnu}.\\
If, in addition, $f$ is convex and $P_R(\cdot)$ is given by \eqnok{def_prob_cvx_gk}, then we have
\beq
\bbe_{\xi_{[N]},R} \left[\Psi(y_{R-1})-\Psi(x^*)\right] \le \bbe_{\xi_{[N]},R} [g_R]
\le  \frac{2\left\{\Psi(x_0)- \Psi^*+\sum_{k=1}^N \left[\frac{7\mu{\cal C}_\nu^2 \alpha_k^{1+2\nu}}{2A_k} + \frac{\sigma^2 \alpha_k}{2\mu b_k A_k}\right]\right\}}{\sum_{k=1}^N \alpha_k  A_k^{-1}}. \label{main_cvx_st-pf-0}
\eeq

\end{itemize}

\end{theorem}

\begin{proof}
We first show part a). Noting optimality condition of \eqnok{def_xk_tilde}, (strong) convexity of $h$, for any $u \in X$ and some $p_h \in \partial h(x_k)$, we have
\begin{align}
&\langle f'(y_{k-1})+\Delta_k+p_h , u-x_k \rangle \ge 0, \nn \\
&h(u) \ge h(x_k) + \langle p_h, u-x_k \rangle +\frac{\mu}{2} \|u-x_k\|^2,\nn
\end{align}
where $\Delta_k= \bar G(y_{k-1},\xi_k)-f'(y_{k-1})$. Summing up the above inequalities, we obtain
\beq \label{opt_stg_h_st}
\langle f'(y_{k-1})+\Delta_k, x_k-u \rangle +h(x_k) +\frac{\mu}{2}\|x_k-u\|^2 \le h(u) \ \ \ \  \forall u \in X.
\eeq
Letting $u=w(f'(y_{k-1}))$ (the solution of the subproblem when true gradient is used) in the above inequality, multiplying it by $\alpha_k$, adding it up with \eqnok{prof_hold0}, and \eqnok{h_strng_2}, we obtain
\beqa
\Psi(y_k) &\le& \Psi(y_{k-1})-\alpha_k \left[\langle f'(y_{k-1}), y_{k-1}-w(f'(y_{k-1})) \rangle +h(y_{k-1})-h(w(f'(y_{k-1})))\right]+ \frac{L_\nu \alpha_k^{1+\nu}}{1+\nu} \|x_k-y_{k-1}\|^{1+\nu} \nn \\ &+&\alpha_k \langle \Delta_k, w(f'(y_{k-1}))-x_k \rangle
- \frac{\mu\alpha_k}{2} \left[\|w(f'(y_{k-1}))-x_k\|^2+(1-\alpha_k)\|x_k-y_{k-1}\|^2\right]\nn \\
&=& \Psi(y_{k-1})-\alpha_k g_k+ \frac{L_\nu \alpha_k^{1+\nu}}{1+\nu} \|x_k-y_{k-1}\|^{1+\nu}+\alpha_k \langle \Delta_k, w(f'(y_{k-1}))-x_k \rangle \nn \\
&-& \frac{\mu\alpha_k}{2} \left[\|w(f'(y_{k-1}))-x_k\|^2+(1-\alpha_k)\|x_k-y_{k-1}\|^2\right]\label{proof_cv_st_gk0} \\
&\le& \Psi(y_{k-1})- \alpha_k g_k + \frac{L_\nu \alpha_k^{1+\nu}}{1+\nu} \left[ \|x_k-y_{k-1}\|^{1+\nu} +\|w(f'(y_{k-1}))-x_k\|^2 \right] +\frac{(1+\nu)\alpha_k^{1-\nu} \|\Delta_k\|^2}{4L_\nu} \nn \\
&\le& \Psi(y_{k-1})- \alpha_k g_k + \frac{2L_\nu \alpha_k^{1+\nu} D_X^2}{1+\nu} +\frac{(1+\nu)\alpha_k^{1-\nu} \|\Delta_k\|^2}{4L_\nu}, \label{proof_cv_st_gk}
\eeqa
where the second inequality follows from the fact that
\beq \label{hold_ineq_delta}
\alpha_k \langle \Delta_k, w(f'(y_{k-1}))-x_k \rangle \le \alpha_k \|\Delta_k\|_* \|x_k-w(f'(y_{k-1})) \|
\le \frac{m_3 \alpha_k^{2m_4}}{2}\|x_k-w(f'(y_{k-1}))\|^2 + \frac{\alpha_k^{2(1-m_4)}}{2m_3}\|\Delta_k\|_*^2,
\eeq
with $m_3 = \frac{2L_\nu}{1+\nu}$, $m_4 = \frac{1+\nu}{2}$, and the last inequality follows from the assumption that $D_X \ge 1$. Taking expectation from both sides of the above inequality, re-arranging the terms, summing them up, diving both sides by $\sum_{k=1}^N \alpha_k$, noting \eqnok{var_red}, and in the view of
\[
\bbe_{\xi_{[N]},R} \left[g_R\right] = \frac{\sum_{k=1}^N \bbe_{\xi_{[N]}} \alpha_k g_k}{\sum_{k=1}^N \alpha_k},
\]
we obtain \eqnok{main_nocvx_st_gk}. If $f$ is convex, then \eqnok{main_cvx_st_gk} follows similarly to the proof of \eqnok{main_cvx_pf2_gk} by noting \eqnok{gk_cvxbnd}, \eqnok{proof_cv_st_gk}, and \eqnok{def_prob_cvx_gk}.

We now show part b). Noting \eqnok{proof_cv_st_gk0}, \eqnok{hold_ineq_delta} ($m_3 =\mu$, $m_4 =1/2$), and \eqnok{prof_hlod_cvx2} ($m_1 =\alpha_k$, $m_2=\sqrt 7$), we have
\beqa
\Psi(y_k) &\le& \Psi(y_{k-1})- \alpha_k g_k - \frac{\alpha_k(6-7\alpha_k)\mu}{14} \|x_k-y_{k-1}\|^2 +\frac{L_\nu^2 \alpha_k^{1+2\nu}}{\mu(1+\nu)^2} \|x_k-y_{k-1}\|^{2\nu} +\frac{\alpha_k \|\Delta_k\|_*^2}{2\mu} \nn \\
&\le& \Psi(y_{k-1})- \alpha_k g_k +\frac{7L_\nu^2 \alpha_k^{1+2\nu}D_X^{2\nu}}{2\mu(1+\nu)^2}  +\frac{\alpha_k \|\Delta_k\|_*^2}{2\mu},
\eeqa
where the second inequality follows from \eqnok{cond_alpha_pf_st} and boundedness of $X$.
Rest of the proof is similar to part a) and hence we skip the details.
\end{proof}
}
\vgap

In the next result, we specialize rate of  convergence of Algorithm~\ref{alg_SCGT} by specifying the choice of stepsizes.
{\color{black}
\begin{corollary} \label{corol_SCGT}
Suppose that the sequence $\{x_k,y_k\}_{k \ge 0}$ is generated by Algorithm~\ref{alg_SCGT} and $X$ is bounded.
\begin{itemize}
\item [a)] If stepsizes, batch sizes, and the iteration limit are set to
\beq \label{alpha_cond_smooth2_gk}
\alpha_k = \frac{1}{N_0(\epsilon)^\frac{1}{1+\nu}}, \ \ b_k = \bar b_\epsilon = \left \lceil\frac{(1+\nu) \sigma^2 D_{f,X}}{L_\nu \epsilon^2} \right \rceil \ \ \forall k \ge 1, \ \ N_0(\epsilon) =\left \lceil \left(\frac{2D_{f,X}}{\epsilon}\right)^\frac{1+\nu}{\nu} \right \rceil
\eeq
for some $D_{f,X} \ge \Psi(x_0)- \Psi^* +(2L_\nu D_X^2)/(1+\nu)$ and a given target accuracy $\epsilon>0$. Then, the total required number of calls to the stochastic oracle to have $\bbe[g_R] \le \epsilon$, where $R$ is uniformly distributed over $\{1,2,\ldots, N_0(\epsilon)\}$, is bounded by
\beq \label{nocvx_bnd_st_gk}
\bar b_\epsilon N_0(\epsilon)  =  \frac{{\cal O} (1) \sigma^2 D_{f,X}^\frac{1+2\nu}{\nu}}{L_\nu \epsilon^\frac{1+3\nu}{\nu}}.
\eeq
If $f$ is convex, the stepsizes are set to \eqnok{cond_alpha_pf2}, $P_R(\cdot)$ is given by \eqnok{def_prob_cvx_gk}, batch sizes, and the iteration limit are set to
\beq \label{alpha_cond_smooth2_gk2}
b_k =  \left \lceil\left(\frac{(1+\nu) \sigma (k+3)^\nu}{2^{2\nu+1.5} L_\nu D_X}\right)^2 \right \rceil \ \ \forall k \ge 1, \ \ N_1(\epsilon) =\left \lceil4 \left(\frac{216 L_\nu D_X^2}{(1+\nu) \epsilon}\right)^\frac{1}{\nu}\right \rceil,
\eeq
and assuming $\Psi(x_0)- \Psi^* \le {\cal O}(N_1(\epsilon)^{3-\nu}L_\nu D_X^2)$ (for simplicity), then the complexity bound in \eqnok{nocvx_bnd_st_gk} is improved to
\beq \label{cvx_bnd_st_gk}
b_\epsilon N_1(\epsilon) = {\cal O} (1) \sigma^2 \left(\frac{L_\nu D_X^{2+2\nu}}{ \epsilon^{1+2\nu}} \right)^\frac{1}{\nu},
\eeq
which also guarantees $\bbe \left[\Psi(y_{R-1})-\Psi(x^*)\right] \le \epsilon$.

\item [b)] If $\mu>0$, stepsizes, batch sizes, and the iteration limit are set to
\beq \label{cond_alpha_pf_st2}
\alpha_k = \frac{6}{7N_2(\epsilon)^\frac{1}{1+2\nu}} \ \ b_k = b_\epsilon = \left \lceil\frac{2\sigma^2}{\mu \epsilon} \right \rceil \ \ \forall k \ge 1, \ \ N_2(\epsilon) =\left \lceil \left(\frac{14}{3\epsilon} \left[\frac{\Psi_0-\Psi^*}{3}+ \mu \left(\frac{6^\nu {\cal C}_\nu}{7^\nu}\right)^2\right]\right)^\frac{1+2\nu}{2\nu}\right \rceil,
\eeq
then the complexity bound in \eqnok{nocvx_bnd_st_gk} is improved to
\beq \label{nocvx_bnd_st2}
b_\epsilon N_2(\epsilon)  = \frac{ {\cal O}(1)\sigma^2}{\mu \epsilon^\frac{1+4\nu}{2\nu}} \left[\Psi(x_0)- \Psi^*+ \frac{L_\nu^2 D_X^{2\nu}}{\mu}\right]^\frac{1+2\nu}{2\nu}.
\eeq
If, $f$ is convex, stepsizes are set to \eqnok{cond_alpha_pf2}, the iteration limit is set to
\beq \label{cond_alpha_pf_st3}
N_3(\epsilon) =\left \lceil \left(\frac{56 \cdot 36^\nu \mu {\cal C}_\nu^2 }{\epsilon}\right)^\frac{1}{2\nu}\right \rceil,
\eeq
and assuming $\Psi(x_0)- \Psi^* \le {\cal O}(\mu {\cal C}_\nu^2 N_3(\epsilon)^{3-2\nu})$ (for simplicity), then the complexity bound in \eqnok{cvx_bnd_st_gk} is improved to
\beq \label{cvx_bnd_st3}
b_\epsilon N_3(\epsilon) = \frac{{\cal O} (1) \sigma^2 L_\nu^\frac{1}{\nu}D_X}{(\mu \epsilon)^\frac{1+2\nu}{2\nu}}.
\eeq
Moreover, this bound guarantees $\bbe \left[\Psi(y_{R-1})-\Psi(x^*)\right] \le \epsilon$.

\end{itemize}

\end{corollary}

\begin{proof}
By \eqnok{def_prob_gk} and choice of $\alpha_k$ in \eqnok{alpha_cond_smooth2_gk} we have $R$ is uniformly distributed and together with \eqnok{main_nocvx_st_gk}, we obtain
\[
\bbe_{\xi_{[N]},R} \left[g_R\right] \le  \frac{\Psi(x_0)- \Psi^*+2L_\nu D_X^2/(1+\nu)}{N_0(\epsilon)^\frac{\nu}{1+\nu}}+\frac{(1+\nu)\sigma^2 N_0(\epsilon)^\frac{\nu}{1+\nu}}{4L_\nu \bar b_\epsilon}.
\]
Replacing choices of batchsizes and the iteration limit given by \eqnok{alpha_cond_smooth2_gk}, in the above relation, we clearly have $\bbe[g_R] \le \epsilon$. Hence, \eqnok{nocvx_bnd_st_gk} follows from the fact that the total number of calls to the $\SO$ is bounded by $\sum_{k=1}^N b_k = \sum_{k=1}^{N_0(\epsilon)} \bar b_\epsilon =  \bar b_\epsilon N_0(\epsilon)$. Moreover, If $f$ is convex, then combining \eqnok{cond_alpha_pf2}, \eqnok{def_prob_cvx_gk}, \eqnok{main_cvx_st_gk}, \eqnok{alpha_cond_smooth2_gk2}, and similarly to the proof of Corollary~\ref{corol_CGT_pf}.a, we obtain \eqnok{cvx_bnd_st_gk} which also guarantees $\bbe \left[\Psi(y_{R-1})-\Psi(x^*)\right] \le \epsilon$.\\
Part b) follows similarly and hence we skip the details.

\end{proof}

\vgap

We now add a few remarks about the above results. First, note that the framework of Algorithm~\ref{alg_MCGT} cannot be directly generalized to the stochastic setting. More specifically, since we do not know the exact value of the objective function $\Psi$, taking that minimum in Step 2 of this algorithm requires much computational effort under stochastic setting. Hence, we use different stepsize policies for both nonconvex and (strongly) convex cases in the above result which lead to different complexity bounds.

Second, the complexity bound in \eqnok{nocvx_bnd_st_gk} seems to be the first one for solving stochastic nonconvex weakly smooth optimization problems obtained by the class of algorithms including subproblems in the form \eqnok{subproblem}. When $\nu=1$, this bound is reduced to ${\cal O}(\epsilon^{-4})$ which is in the same order of the one achieved by a stochastic CG method in \cite{ReSrPoSm16} for smooth nonconvex problems ($h \equiv 0$). Moreover, when $h$ is strongly convex, \eqnok{nocvx_bnd_st_gk} is improved to \eqnok{nocvx_bnd_st2} which is also better than the complexity of ${\cal O}(\epsilon^{-\tfrac{(1+\nu)^2}{\nu}})$ obtained for a gradient type method including subproblems in the form \eqnok{subproblem2} in which the Bregman distance function is replaced by $\|\cdot\|^{1+\nu}$ and under the boundedness assumption of higher momentums of the stochastic error in gradient estimations (\cite{JiaZha14}).

Third, when $f$ is convex, \eqnok{cvx_bnd_st_gk} seems to be the first complexity bound of CGT methods for solving stochastic convex weakly smooth optimization problems. When $\nu=1$, this bound is ${\cal O}(\epsilon^{-3})$ which is worse than ${\cal O}(\epsilon^{-2})$ obtained in \cite{LanZho16} for a conditional gradient method skipping computing stochastic gradients for some iterations. When $\Psi$ is strongly convex, \eqnok{cvx_bnd_st_gk} is improved to \eqnok{cvx_bnd_st3} and it reduces to ${\cal O}(\epsilon^{-1.5})$ for solving stochastic strongly convex smooth optimization problems. Note however, that this bound is still worse than the optimal complexity ${\cal O}(\epsilon^{-1})$ for stochastic first-order methods. In the rest of this section, we provide a different choice of stepsizes for Algorithm~\ref{alg_SCGT} to achieve the aforementioned optimal complexity bound when $\nu=1$ and $\mu>0$.

\vgap

\begin{theorem}
Suppose that the sequence $\{x_k,y_k\}_{k \ge 0}$ is generated by Algorithm~\ref{alg_SCGT}, $h$ is strongly convex with parameter $\mu>0$, $f$ has Lipschitz continuous gradient i.e., $\nu=1$ in \eqnok{f_holder1} and \eqnok{f_holder2}, the stepsizes satisfy
\beq \label{alpha_cond_smooth}
\alpha_k \le \frac{\mu}{L_1+\mu} \ \ \forall k \ge 1,
\eeq
and $P_R(\cdot)$ is given by \eqnok{def_prob_gk}. Then, we have
\beq \label{main_nocvx_st}
\bbe_{\xi_{[N]},R} \left[g_R\right] \le  \frac{\Psi(x_0)- \Psi^*+\frac{\sigma^2}{2\mu}\sum_{k=1}^N \frac{\alpha_k}{b_k}}{\sum_{k=1}^N \alpha_k}
\eeq
If, in addition, $f$ is convex and $P_R(\cdot)$ is given by \eqnok{def_prob_cvx_gk}, then we have
\beq
\bbe_{\xi_{[N]},R} \left[\Psi(y_{R-1})-\Psi(x^*)\right] \le \bbe_{\xi_{[N]},R}[g_R]
\le \frac{2[\Psi(x_0)- \Psi^*]+\frac{\sigma^2}{\mu}\sum_{k=1}^N \frac{\alpha_k}{b_k A_k}}{\sum_{k=1}^N \alpha_k A_k^{-1}}, \label{main_cvx_grad_st-0}
\eeq
where $A_k$ is defined in \eqnok{def_Alpha}

\end{theorem}

\begin{proof}
By \eqnok{proof_cv_st_gk0} and \eqnok{hold_ineq_delta} (with $m_3=\mu$, $m_4 =1/2$), we have
\beq
\Psi(y_k) \le \Psi(y_{k-1})- \alpha_k g_k -\frac{\alpha_k[\mu -(\mu+L_1) \alpha_k]}{2} \|x_k-y_{k-1}\|^2 + \frac{\alpha_k \|\Delta_k\|_*^2}{2\mu} \le \Psi(y_{k-1})- \alpha_k g_k + \frac{\alpha_k \|\Delta_k\|_*^2}{2\mu},\nn
\eeq
where the second inequality follows from \eqnok{alpha_cond_smooth}.
Rest of the proof is similar to that of Theorem~\ref{theo_SCGT}.a) and hence we skip the details.

\end{proof}

\vgap

In the next result, we specify the appropriate choice of stepsizes for Algorithm~\ref{alg_SCGT} when $f$ is smooth.

\begin{corollary}
Suppose that the sequence $\{x_k,y_k\}_{k \ge 0}$ is generated by Algorithm~\ref{alg_SCGT}, $h$ is strongly convex with $\mu>0$, and $f$ has Lipschitz continuous gradient. If batch sizes are set to \eqnok{cond_alpha_pf_st2}, stepsizes, and the iteration limit are set to
\beq \label{alpha_cond_smooth4}
\alpha_k = \bar \alpha= \frac{\mu}{L_1+\mu}, \ \ \forall k \ge 1, \ \ N_4(\epsilon) =\left \lceil\frac{4 (\mu+L_1)[\Psi(x_0)- \Psi^*]}{3\mu \epsilon} \right \rceil
\eeq
for a given $\epsilon>0$, then the total required number of calls to the stochastic oracle to have $\bbe[g_R] \le \epsilon$, where $R$ is uniformly distributed over $\{1,2,\ldots, N_4(\epsilon)\}$, is bounded by
\beq \label{nocvx_bnd_st}
b_\epsilon N_4(\epsilon)  = {\cal O} (1) \left(\frac{(\mu+L_1)\sigma^2 [\Psi(x_0)- \Psi^*]}{\mu^2 \epsilon^2} \right).
\eeq
If, in addition, $f$ is convex, $P_R(\cdot)$ is given by \eqnok{def_prob_cvx_gk}, and the iteration limit is set to
\beq \label{def_N2}
N_5(\epsilon) =\left \lceil\frac{2(\mu+L_1)}{\mu} \log \left(\frac{4(\mu+L_1)[\Psi(x_0)- \Psi^*]}{\mu \epsilon} \right) \right \rceil,
\eeq
then the complexity bound in \eqnok{nocvx_bnd_st} is improved to
\beq \label{cvx_bnd_st}
b_\epsilon N_5(\epsilon) = {\cal O} (1) \left(\frac{(\mu+L_1)\sigma^2}{\mu^2 \epsilon} \log \left[\frac{(\mu+L_1)[\Psi(x_0)- \Psi(x^*)]}{\mu \epsilon} \right] \right).
\eeq
Moreover, the above bound guarantees $\bbe \left[\Psi(y_{R-1})-\Psi(x^*)\right] \le \epsilon$.

\end{corollary}

\begin{proof}
The results are followed similarly to that of Corollary~\ref{corol_SCGT} by noting \eqnok{alpha_cond_smooth4} and \eqnok{def_Alpha2}.
\end{proof}

}

\vgap

Note that, the bound in \eqnok{nocvx_bnd_st} is better than the one in \eqnok{nocvx_bnd_st2} when $\nu=1$ and after disregarding $\mu$, it is the best known complexity of first-order methods when applied to smooth nonconvex stochastic programming (see e.g., \cite{GhaLan12-2,GhaLanZhang14,GhaLan15}). Moreover, when $f$ is convex, the complexity bound in \eqnok{cvx_bnd_st} is also better than the one in \eqnok{cvx_bnd_st3} and it is indeed optimal up to a logarithmic constant for solving smooth strongly convex stochastic optimization \cite{nemyud:83}.

\vgap

It should be also pointed out that while all of the complexity bounds in this section are presented for expectation results, we can simply use Markov's inequality to provide similar bounds to have $\Prob \left\{g_k \ge \epsilon\right\} \le \Lambda$ in a single run of these algorithms. In particular, these large-deviation bounds are obtained with replacing $\epsilon$ by $\epsilon \Lambda$ in their associated expectation complexity bounds. However, we can improve dependence of these bounds on $\Lambda$ by designing a two-phase variant of the RSCGT method which consists of several runs of the RSCGT method in the first phase and choosing the best solution among these runs as the output in the second phase (see e.g., \cite{GhaLan12-2,GhaLanZha15} for similar two-phase variants of randomized stochastic (projected) gradient methods for smooth problems in more details).

\section{Concluding remarks} \label{conl_rem}
In this paper, we present conditional gradient type methods for solving a class of composite nonlinear optimization problems defined in \eqnok{NLP}. We present a unified analysis for both nonconvex and (strongly) convex problems in the sense that our proposed method achieves the best-known complexity for nonconvex problems, and its rate of convergence is (nearly) optimal if $f$ is convex. We also present a few variants of this method which do not require the knowledge of problem parameters. Since there is no stepsize involved in the subproblems of these methods, their computational efforts are much cheaper per iteration than those of the existing parameter free proximal type methods. Moreover, we consider problem \eqnok{NLP} under stochastic setting where only noisy information of $f$ is available and provide complexity of our generalized methods to this setting. To the best of our knowledge, this is the first time that such complexity results of CGT methods are presented for stochastic weakly smooth optimization.

{\color{black}
{\bf Acknowledgement:}
The author is very grateful to the associate editor and the anonymous referees
for their valuable comments for improving the quality and exposition
of the paper.
}

\newcommand{\noopsort}[1]{} \newcommand{\printfirst}[2]{#1}
  \newcommand{\singleletter}[1]{#1} \newcommand{\switchargs}[2]{#2#1}

\end{document}